\documentclass{conm-p-l}

\usepackage{amsmath, amsthm, amssymb}

\newtheorem{theorem}{Theorem}[section]
\newtheorem{lemma}[theorem]{Lemma}

\theoremstyle{definition}

\newtheorem{corollary}[theorem]{Corollary}

\theoremstyle{remark}

\numberwithin{equation}{section}

\newcommand{\Z}{\mathbb Z}
\newcommand{\F}{\mathbb F}
\newcommand{\C}{\mathbf C}
\newcommand{\Q}{\mathbb  Q}

\newcommand{\Fq}{\F_q}

\usepackage{tikz}
\usetikzlibrary{matrix,arrows}

\newcommand{\ka}{{\kappa}}
\newcommand{\cO}{{\mathcal O}}
\newcommand{\cl}{\mathcal{C}l}
\newcommand{\p}{\mathbf{Pic}}
\newcommand{\N}{\mathcal N}
\newcommand{\m}{\mathfrak m}
\newcommand{\f}{\mathfrak f}
\newcommand{\e}{\mathfrak e}
\newcommand{\g}{\mathfrak g}
\newcommand{\h}{\mathfrak h}

\newcommand{\q}{\mathfrak q}
\newcommand{\pe}{\mathfrak p}
\newcommand{\ma}{\mathfrak a}

\title{Computing Class Groups of Function Fields Using Stark Units}
\author{Ming-Deh Huang and Anand Kumar Narayanan}
\begin{document}
\maketitle
\begin{abstract}
Let $k$ be a fixed finite geometric extension of the rational function field $\mathbb{F}_q(t)$. Let $F/k$ be a finite abelian extension such that there is an $\Fq$-rational place $\infty$ in $k$ which splits in $F/k$ and let $\mathcal{O}_F$ denote the integral closure in $F$ of the ring of functions in $k$ that are regular outside $\infty$.
We describe algorithms for computing the divisor class number and in certain cases for computing the structure of the divisor class group and discrete logarithms between Galois conjugate divisors in the divisor class group of $F$. The algorithms are efficient when $F$ is a narrow ray class field or a small index subextension of a narrow ray class field.\\ \\
We prove that for all prime $\ell$ not dividing $q(q-1)[F:k]$, the structure of the $\ell$-part of the ideal class group $\p(\cO_F)$ of $\mathcal{O}_F$ is determined by Kolyvagin derivative classes that are constructed out of Euler systems associated with Stark units. This leads to an algorithm to compute the structure of the $\ell$ primary part of the divisor class group of a narrow ray class field for all primes $\ell$ not dividing $q(q-1)[F:k]$.
\end{abstract}
\section{Introduction}
\noindent Fix $k/\F_{q}(t)$, a finite geometric extension of the rational function field $\mathbb{F}_q(t)$. Let $F/k$ be a finite abelian extension of conductor $\m$ such that $F$ has an unramified $\Fq$ rational place $\mathfrak{B}$. Let $\infty$ be the place in $k$ lying below $\mathfrak{B}$.  Since $\infty$ splits completely in $F/k$, we call $F$ totally real with respect to $\infty$. Denote by $\cO_k$ the ring of functions in $k$ regular outside $\infty$.
Let $G:=Gal(F/k)$ denote the Galois group of the extension, $\deg(\f)$ the degree of an ideal $\f \subset \cO_k$ as the degree of the divisor $\f$ and $\cO_F$ the integral closure of $\cO_k$ in $F$. Let $H_\m$ denote the narrow ray class field of modulus $\m \subset \cO_K$.\\ \\
Let $\mathcal{D}_F$ denote the group of divisors of $F$, which is the free abelian group on the places of $F$. Denote by $\mathcal{D}^0_F$ the subgroup of $\mathcal{D}_F$ of degree zero divisors and by $\mathcal{P}_F$ the subgroup of principal divisors which consists of divisors of functions in $F$. The quotient $\mathcal{C}l^0_F = \mathcal{D}^0_F/\mathcal{P}_F$ is the (degree zero) divisor class group of $F$.\\ \\
The divisor class group $\mathcal{C}l^0_F$ is a finite abelian group and fits in the following exact sequence \cite{ros}[Prop 14.1]  $$ 0 \longrightarrow R_F \longrightarrow  \mathcal{C}l^0_F \longrightarrow \p(\mathcal{O}_F) \longrightarrow 0$$ where $R_F$ is the regulator and $\p(\mathcal{O}_F)$ is the ideal class group of $\mathcal{O}_F$. The regulator $R_F$ is the quotient of the group of degree zero divisors supported at the places in $F$ above $\infty$ by the group of principal divisors supported at the places in $F$ above $\infty$. Let $h(F)$ denote $|\mathcal{C}l^0_F|$ and let $h(\cO_F)$ denote $|\p(\mathcal{O}_F)|$.\\ \\
Stark units are certain functions in $F$ related to special values of Artin's L-functions and appear in the context of the Brumer-Stark and Stark conjectures (see \S \ref{notation} for the definition). They are supported at the places dividing the conductor $\m$ and the places dividing $\infty$.
Analytic class number formula relates the divisor class number $h(F)$,  which is the order of $\mathcal{C}l^0_F$, to certain special values of the Artin L-functions associated with the non-trivial irreducible characters of $G$. These special values are related to Stark units and in section \ref{divisor_class_number} we describe this correspondence and develop an algorithm to compute the divisor class number given a certain $\Z[G]$ generator of the Stark units. The following theorem is proven in section \ref{divisor_class_number}.
\begin{theorem}\label{ray_class_theorem}
There is a deterministic algorithm that given the rational function field $\F_q(t)$, a totally real finite abelian extension $F/\F_q(t)$ presented as an irreducible polynomial $X_F(y) \in \F_q(t)[y]$ such that $F = \F_q(t)[y]/(X_F(y))$ and a generator for the conductor $\m$ of $F$, computes the degree zero divisor class number $h(F)$ in time polynomial in $q^{\deg(\m)}$ and the size of $X_F$.
\end{theorem}
\noindent Theorems 1.1,1.2 and 1.3 assume that a generator of the conductor of $F$ is given. In \S~\ref{conductor}, we sketch how to compute the conductor of $F$ given an irreducible polynomial $X_F(y) \in \F_q(t)[y]$ such that $F = \F_q(t)[y]/(X_F(y))$. \\ \\
All stated algorithmic results that take $\F_q(t)$ as an input assume that an efficient representation of the finite field $\F_q$ if given. An efficient representation is one that allows field addition and multiplication in time polylogarithmic in the field size (see \cite{len} for a formal definition).\\ \\
The algorithm in theorem \ref{ray_class_theorem} is efficient when $F$ is of small index in $H_\m$. For instance, when $k =\F_q(t)$ and $F$ is $H_\m$,  the running time is polylogarithmic in $h(H_\m)$. This is because the genus $g(H_\m)$ of $H_\m$ grows roughly as $|(\cO_k/\m)^\times| \log(|(\cO_k/\m)^\times|)$ (see \cite{hay_exp} and \cite{sal}[Thm 12.7.2] for an exact expression), $|(\cO_k/\m)^\times|$ is about $q^{\deg \m}$ and the divisor class number $h(H_\m)$ is approximately $q^{g(H_\m)}$.\\ \\
Due to Lauder and Wan \cite{lw}[Theorem 37], there is an algorithm to compute the divisor class number of an arbitrary finite extension of $\F_q(t)$ in time polylogarithmic in the divisor class number if the characteristic of $k$ is fixed.
It would be interesting to compare the performance the algorithm of Lauder and Wan when restricted to the family of narrow ray class fields with the algorithm in theorem \ref{ray_class_theorem}.\\ \\ 
In \cite{yin}, Yin defined an ideal $I_F$ in $\Z[G]$ that annihilates $\mathcal{C}l^0_F$. The ideal is comprised of Stickelberger elements that arise in the proofs by Deligne \cite{del} and Hayes \cite{hay} of the function field analogue of the Brumer-Stark conjecture, and are intimately related to Stark units.   When $F$ is either $K_\m$, the cyclotomic extension of conductor $\m$ or $H_\m \subset K_\m$, the narrow ray class field of modulus $\m$, Yin \cite{yin} derived an index theorem demonstrating that $[\Z[G]:I_{F}]$ is up to a power of $q-1$, the degree zero divisor class number of $F$.  Ahn,  Bae and Jung \cite{ahn} extended the index theorem to all sub extensions of $K_\m$. Since a totally real extension of conductor $\m$ is contained in $H_\m$, the index theorem applies to $F$ that we consider.\\ \\
The ideal $I_F$ and the corresponding index theorem are analogues of the Stickelberger ideal in cyclotomic extensions of $\Q$ and the Iwasawa-Sinnott \cite{sin} index formula. It is remarkable that the index of $I_F$ relates to the divisor class number in its entirety. In contrast, in cyclotomic extensions over $\Q$, the index of the Stickelberger ideal relates only to the relative part of the class number. As Yin \cite[\S~1]{yin} suggests, it is perhaps appropriate to regard $[\Z[G]:I_F]$ as being composed of both the relative part which is analogous to the Iwasawa-Sinnott index and the real part which corresponds to the Kummer-Sinnott \cite{sin} unit index formula. The construction of a large ideal such as $I_F$ that annihilates the divisor class group is possible in part due to the partial zeta functions over function fields being $\Q$ valued when evaluated at $0$. In contrast, in cyclotomic extensions over $\Q$, the evaluation of partial zeta functions of the real part of cyclotomic extensions at $0$ could be irrational and the Stickelberger ideal corresponding to the real part is the zero ideal.\\ \\
Based on the construction in \cite{yin,ahn}, the following theorem is proven in \S~\ref{stickelberger}.
\begin{theorem}\label{stickelberger_generating_set}
There is a deterministic algorithm that given the rational function field $\F_q(t)$, a totally real finite abelian extension $F/\F_q(t)$ presented as an irreducible polynomial $X_F(y) \in \F_q(t)[y]$ such that $F = \F_q(t)[y]/(X_F(y))$ and a generator for the conductor $\m$ of $F$, computes a generating set of $I_F$ in time polynomial in $q^{\deg(\m)}$ and the size of $X_F$.
\end{theorem}
\noindent Let $r_F$ be the largest factor of $h(F)$ that is relatively prime to $h(k)[F:k]$. If the $r_F$-torsion $\cl^0_F[r_F]$ of $\cl^0_F$ is $\Z[G]$ cyclic, then the structure of $\cl^0_F[r_F]$ is determined by the Stickelberger ideal $I_F$. This leads to an algorithm to compute the structure of $\cl^0_F$ resulting in the following theorem proven in \S~\ref{stickelberger_structure}.
\begin{theorem} \label{divisor_class_structure} There is a deterministic algorithm that given the rational function field $\F_q(t)$, a totally real finite abelian extension $F/\F_q(t)$ presented as an irreducible polynomial $X_F(y) \in \F_q(t)[y]$ such that $F = \F_q(t)[y]/(X_F(y))$ and a generator for the conductor $\m$ of $F$, if $\cl^0_F[r_F]$ is $\Z[G]$ cyclic, computes the structure of $\mathcal{C}l^0_F$ in time polynomial in $q^{\deg(\m)}$ and the size of $X_F$. If in addition a $\Z[G]$ generator of $\cl^0_F[r_F]$ is given, the invariant factor decomposition of $\cl^0_F$ can be computed in time polynomial in $q^{\deg(\m)}$ and the size of $X_F$.
\end{theorem}
\noindent Further, given a $\Z[G]$ generator $\gamma$ of $\cl^0_F[r_F]$, we can project an element in $\Z[G] \gamma$ into the invariant decomposition of $\cl^0_F$ and hence efficiently compute discrete logarithms between $\gamma$ and its Galois conjugates. See \S \ref{stickelberger_structure}\ for details.\\ \\
In section \ref{regulator}, we describe an algorithm to compute the structure of the $\ell$ primary part of the regulator $R_F$ for a prime $\ell$ not dividing $q(q-1)[F:k]$ and prove the below theorem and the corollary that follows.
\begin{theorem}\label{regulator_theorem} There is a deterministic algorithm that given the rational function field $\F_q(t)$, a totally real finite abelian extension $F/\F_q(t)$ presented as an irreducible polynomial $X_F(y) \in \F_q(t)[y]$ such that $F = \F_q(t)[y]/(X_F(y))$ and a prime $\ell \nmid q(q-1)[F:k]$, computes the structure of the $\ell$-primary part of the regulator $R_F$ in time polynomial in $\log(q)$ and the size of $X_F$.
\end{theorem}
\begin{corollary}\label{ell ideal class number}
There is a deterministic algorithm that given the rational function field $\F_q(t)$, a totally real finite abelian extension $F/\F_q(t)$ presented as an irreducible polynomial $X_F(y) \in \F_q(t)[y]$ such that $F = \F_q(t)[y]/(X_F(y))$ and a prime $\ell \nmid q(q-1)[F:k]$ computes the cardinality of the $\ell$-primary part of the ideal class group $\p(\cO_F)$ in time polynomial in $\log(q)$ and the size of $X_F$.
\end{corollary}
\noindent By factoring $h(F)$, we can obtain a list of primes that contain all primes dividing $h(\cO_F)$. If $(h(F),q[F:k])=1$, by determining the cardinality of the $\ell$-primary part of $\p(\cO_F)$ for every prime $\ell$ in the list, we can determine $h(\cO_F)$. Since $h(F) = \tilde{\mathcal{O}}(q^g)$ where $g$ is the genus of $F$, factoring $h(F)$ using the Number Field Sieve takes $\tilde{\mathcal{O}}(q^{g^{1/3}})$ time under heuristics \cite{nfs}.\\ \\
\noindent In particular if $\ \p(\cO_F)$  is trivial then $\mathcal{C}l^0_F$ is isomorphic to $R_F$, which is $\Z[G]$-cyclic with any prime over $\infty$ as a generator and we can apply Theorem \ref{divisor_class_structure}.\\ \\
If $\p(\cO_F)$ is not trivial, then we deploy the full machinery of the Euler systems of Stark units to determine the $\ell$-part of $\p(\cO_F)$ for $\ell$ not dividing $q[F:k]$.   We briefly describe this approach below.\\ \\
The intersection of the multiplicative group $S_F$ generated by the Stark units with the unit group $\cO_F^\times$ is of finite index in $\cO_F^\times$ and this index equals $|\p(\cO_F)|$  by the Kummer-Sinnott unit index formula.
Gras conjecture \cite{gra} over function fields is a refinement of the Kummer-Sinnott index formula and it relates the cardinality of the $\ell$-primary part of $\p(\cO_F)$ to the (finite) index $[\cO_F^\times:\cO_F^\times\cap S_F]$, where $S_F$ is the group of Stark units. If $\ell$ is a prime not dividing $q[F:k]$, $\chi$ a non trivial irreducible $\Z_\ell$ character of $G$ and $e(\chi) \in \Z_\ell [G]$ the corresponding idempotent, then Gras conjecture claims
$$ |e(\chi)(\Z_\ell \otimes_\Z (\cO_F^\times/\cO_F^\times \cap S_F)))|= |e(\chi)(\Z_\ell \otimes_\Z \p(\cO_F))|.$$
In its originally formulated context of cyclotomic extensions of $\Q$, Gras conjecture is known to be true as a consequence of the proof of the main conjecture of Iwasawa theory by Mazur and Wiles \cite{mw}. Kolyvagin \cite{kol} gave a more elementary proof as an application of the method of Euler systems that he had just developed. Rubin furthered the theory of Euler systems and using it proved the main conjecture of Iwasawa theory over imaginary quadratic extensions \cite{rub}\cite{rub_book}.\\ \\
Feng and Xu \cite{fen} introduced the method of Euler systems to the function field setting and proved the Gras conjecture when $F=H_\m$ and $k$ is the rational function field. In a recent work \cite[Thm~1.1]{ouk}, Oukhaba and Viguie extended the proof to all totally real abelian extensions $F$ except when $\ell \mid [H_\m:F]$ and $F$ contains the $\ell^{th}$ roots of unity.

\ \\Given $N$ that is a power of a prime $\ell$ not dividing $q[F:k]$, we consider Euler systems of modulus $N$ that starts with an element of $S_F$.  From an Euler system $\Psi$, we have a derived system, called a Kolyvagin system, consisting of of functions $\ka(\ma) \in F^\times$ indexed by $\ma \in B_N$, where $B_N$ is the set of all finite square free product of a certain infinite set of primes in $\cO_k$. The places that appear in the divisor $[\ka(\ma)]$ admit a precise characterization up to an $N^{th}$ multiple due to the properties of Euler systems. For prime $v \nmid \ma$, $[\ka(\ma)]_v=0\mod N$, where $[\ka(\ma)]_v$ denote the projection of $[\ka(\ma)]$ to the primes of $F$ that lie above $v$. For prime $v | \ma$, the relation between $\ka(\ma)$ and
$\ka(\ma/v)$ is governed by a Galois equivariant map $\varphi_v$  where $[\ka (\ma)]_v =\varphi_v (\ka(\ma/v)) \mod N$.

\ \\For an $\alpha\in F^{\times}/F^{\times N}$, if $e(\chi) \alpha$ is an $\ell^c$-th power but not an $\ell^{c+1}$-th power in $F^\times/F^{\times N}(\chi)$, we define $\ell^c$ to be the $\chi$-index of $\alpha$.
We consider a Kolyvagin system of modulus $N$ divisible by $\ell t^2$ starting from a unit $\kappa(\e) \in E$ of
$\chi$-index $t$, where
$t^{\dim(\chi)}$ is the cardinality of the $\chi$-component of
$\cO_F^\times/(\cO_F^\times \cap S_F)$.
Write $\ka(\ma) \stackrel{\chi}{\to} \ka(\ma\pe)$, if
$u T \mathfrak{P} = [e(\chi) \ka(\ma\pe)]_\pe \mod N $ where $T$ is the $\chi$-index of $\ka(\ma)$,
$\mathfrak{P}$ is a prime over $\pe$ in $F$ and
$u\in ((\Z/N\Z [G])(\chi))^{\times}$.
A $\chi$-path starting from $\ka(\e)$:
\[ \ka(\e) \stackrel{\chi}{\to} \ka(\pe_1) \stackrel{\chi}{\to} \ka(\pe_1\pe_2) \stackrel{\chi}{\to} \ldots \stackrel{\chi}{\to} \ka(\pe_1 \pe_2 \ldots \pe_n)\]
is {\em complete} if the $\chi$-index of the last node $\ka(\pe_1 \pe_2 \ldots \pe_n)$ is $1$.
The following theorem (proven in \S~\ref{structure}) says that the $\chi$-component of $\p(\cO_F)$ is completely determined by a complete $\chi$-path.
\begin{theorem}
\label{chi-structure}
Let $\ka(\e) \stackrel{\chi}{\to} \ka(\pe_1) \stackrel{\chi}{\to} \ka(\pe_1\pe_2) \stackrel{\chi}{\to} \ldots \stackrel{\chi}{\to} \ka(\pe_1 \pe_2 \ldots \pe_n)$ be a $\chi$-path from $\ka(\e)$.
Let $\C_i$ be the subgroup of $\p(\cO_F)$ generated by all primes of $F$ above $\pe_1$, $\pe_2$,... $\pe_i$.
Let $t_i$ be the $\chi$-index of $\ka(\pe_1\pe_2 \ldots \pe_i)$. Then $[ \C_i(\chi) : \C_{i-1}(\chi) ] = (t_{i-1} /t_{i})^d$ for $i=1,...,n$, and $t_n=1$ if and only if $\C_n(\chi) = \p_\ell(\cO_F)(\chi)$.  Moreover any $\chi$-path from $\ka(\e)$ can be extended to a complete $\chi$-path.
\end{theorem}

\ \\The characterization of Theorem~\ref{chi-structure} leads to an algorithm that determines the structure of $\p_\ell(\cO_F)(\chi)$ as a  $\Z_\ell[G]$-module,
and the following theorem is proven in \S~\ref{ideal_class_structure}.
\begin{theorem}\label{class_group_number}
Let $H_\m$ denote the narrow ray class field $H_\m$ of conductor $\m$ over the rational function field $k=\F_q(t)$ and for a prime $\ell$ and a $\Z_\ell$ representation $\chi$, let $t(\chi)$ denote the exponent of the $\chi$ component of $\cO_{H_\m}^\times/(\cO_{H_\m}^\times \cap S_{H_\m})$. There is a deterministic algorithm, that given a generator for an ideal $\m$ in $\F_q[t]$ and a prime $\ell$ not dividing $q(q-1)[H_\m:k]$, finds the structure of $\p_\ell(\cO_{H_\m})(\chi)$  as a $\Z_\ell[Gal(H_\m/k)]$ module in time polynomial in $q^{t(\chi)^2}$ and $[H_\m:k]$ for every non trivial irreducible $\mathbb{Z}_l$ representation $\chi$ of $Gal(H_\m/k)$.
\end{theorem}
\noindent Since $t(\chi)^{\dim(\chi)}$ is the cardinality of the $\chi$ component of $\cO_{H_\m}^\times/(\cO_{H_\m}^\times \cap S_{H_\m})$, by Gras' conjecture $t(\chi)^{\dim(\chi)} = |\p_\ell(\cO_{H_\m})(\chi)|$. The exponential dependence on the size of $\p_\ell(\cO_{H_\m})(\chi)$ would be less of a concern if the regulator is expected to be large and the ideal class group is expected to be small.  The function field analog of the Cohen-Lenstra heuristics \cite{fw,achter} conjecture that an isomorphism class of an abelian $\ell-$group $H$ occurs as the $\ell$ primary part of the divisor class group of function fields with probability $$\frac{1}{Aut(H)} \prod_{i=1}^{\infty} (1 - \ell^{-i}).$$
In particular, it predicts that the $\ell-$primary part of the divisor class group is more likely to be cyclic than otherwise. If the point $\infty$ is chosen at random from the rational places that split completely in $F$, then it is likely that aforementioned cyclic subgroup is contained in the subgroup generated by the places in $F$ above $\infty$. Thus the regulator is expected to be large and the ideal class group is expected to be small.

\section{Stickelberger Elements and Stark Units}\label{notation}
\subsection{Cyclotomic Extensions}
In this subsection, we build notation and recount properties of cyclotomic function fields over global function fields based on the theory of sign-normalized Drinfeld modules developed by Hayes. Refer to \cite{hay} for a detailed description and proofs of claims made here.\\ \\
Let $k_\infty$ be a completion of $k$ at $\infty$. Let $\F(\infty)$ be the constant field of $k_\infty$ and $\Omega$ the completion of an algebraic closure of $k_\infty$. Let $\mathcal{V}_\infty$ be the extension of the normalized valuation of $k_\infty$ at $\infty$ to $\Omega$. Fix a sign function $sgn: k_\infty^\times\longrightarrow \F(\infty)^\times$, a co-section of the inclusion morphism $\F(\infty)^\times \hookrightarrow k_\infty^\times$ such that $sgn(\zeta)=1$ for every $\zeta$ in the group of $1$ units $U_\infty^{(1)}$. Let $\rho:\cO_k \longrightarrow \Omega \langle \tau_q \rangle$ be a $sgn$-normalized rank $1$ Drinfeld-module. Here, $\Omega \langle \tau_q \rangle$ is the left twisted polynomial ring with $\tau_q$ satisfying the relation $\tau_q x = x^q \tau_q, \forall x \in \Omega$. The image of $a\in \cO_k$ under $\rho$ is denoted by $\rho_a$.\\ \\
Let $H_\e$ denote the maximal real unramified abelian extension of $k$ and $K_\e$, the normalizing field with respect to $sgn$ obtained by adjoining to $k$ the coefficients of $\rho_a$ for every $a \in \cO_k$. The extension $H_\e/k$ is contained in $K_\e/k$  with $[K_\e:H_\e]$ equalling $q-1$ and the primes in $H_\e$ above $\infty$ are totally ramified in $K_\e/H_\e$.\\ \\
For an integral ideal $\mathfrak{m} \subset \cO_k$, define the $\mathfrak{m}$-torsion points $\Lambda_{\rho}[\mathfrak{m}]:=\{w \in \Omega | \rho_a(w)=0, \forall a\in \mathfrak{m}\}$. As $\cO_k$-modules, $\Lambda_\rho[\m]$ is cyclic and isomorphic to $\cO_k/\m \cO_k$. We fix a generator $\lambda_\mathfrak{m} \in \Lambda_\rho[\mathfrak{m}]$ of $\Lambda_\rho[\m]$ as an $\cO_k$-module as described below.\\ \\
Associated with every rank 1 $\cO_k$-submodule of $\Omega \langle \tau_q \rangle$ is a rank $1$ Drinfeld-module \cite[\S~4]{hay_class}. Let $\xi(\m) \in \Omega$ denote the invariant determined up to an $\F_{q^d_\infty}^\times$ multiple by the property that the $\cO_k$-submodule $\xi(\m)\m$ corresponds to a $sgn$-normalized Drinfeld-module. Define $e_\m(z):= z \prod_{0\neq \gamma \in \m}(\frac{1-z}{\gamma})$ to be the exponential function associated with the $\cO_k$-submodule $\m$. Then $\xi(\m)e_\m(1)$ is determined by $\m$ and $sgn$ up to an $\F_{q^{d_\infty}}^\times$ multiple and generates $\Lambda_{\rho}[\m]$ as an $\cO_k$-module. Set $\lambda_\m := \xi(\m)e_\m(1)$.\\ \\ 
The cyclotomic function field $K_\mathfrak{m}$ is obtained by adjoining $\Lambda_{\rho}[\mathfrak{m}]$ to $K_\e$. Since $\lambda_\mathfrak{m}$ generates $\Lambda_\rho[\mathfrak{m}]$ as an $\cO_k$-module, $K_\mathfrak{m}=K_\e(\lambda_\mathfrak{m})$. The extension $K_\mathfrak{m}/k$ is abelian. For the class of $a \in \cO_k$ in $(\cO_k/\mathfrak{m})^\times$, there is a unique $\sigma_a \in Gal(K_\m/K_\e)$ such that $\sigma_a(\lambda_\mathfrak{m})= \rho_a(\lambda_\mathfrak{m})$ and thus $Gal(K_\m/K_\e) \cong (\cO_k/\mathfrak{m})^\times$. The maximal real subfield of $K_\mathfrak{m}$ denoted by $H_\m$ is the ray class field modulo $\m$ and is independent of $sgn$. Further,  $H_\m=H_\e(\lambda_\m^{q-1})$ and $[K_\m:H_\m] = q-1$.\\ \\
From now on, let $F/k$ be a finite abelian extension of conductor $\m$ and Galois group $G:=Gal(F/k)$.  Let $\e= \cO_k$ denote the unit ideal. For a non zero ideal $\f  \subseteq \cO_k$, define $F_\f:= K_\f \cap F$, $F_\f^+:= H_\f \cap F$ and for $\f \neq \e$, $\lambda_{\f,F}:=\mathcal{N}_{K_\f/F_\f}(\lambda_\f)$. \\ \\
For a finite Galois extension $L/k$, let $\cO_L$ denote the ring of integers of $L$. Let $\cl^0_L$ and $\p(\cO_L)$ refer to the degree zero divisor class group  of $L$ and the ideal class group of $\cO_L$ respectively. For $f \in L$, let $[f]_L$ denote its divisor. If the field in question is clear from the context, we will drop the subscript $L$. For a sub extension $\bar{L}/k \subseteq L/k$ and an integral ideal $\mathfrak{a} \subseteq \cO_{\bar{L}}$, let  $(\mathfrak{a},L/\bar{L}) \in Gal(L/\bar{L})$ denote the Artin symbol and $\mathcal{N}_{L/\bar{L}}$ the norm map. The cardinality of the residue class ring $\cO_{\bar{L}}/\mathfrak{a}$ is denoted by $\mathcal{N}(\mathfrak{a})$. For an integer $b$, let $\mu_b$ indicate the group of $b^{th}$-roots of unity.
\subsection{Elliptic and Stark Units}
Let $B_F$ be the $\Z[Gal(F_\e/k)]$-submodule of $F_\e^\times$ generated by $\F_q^\times$ and $\{\lambda_{\f,F}\}_\f$, where $\f$ ranges over non zero integral ideals of $\cO_k$. The group of elliptic units $E_F$, which is of finite index in $\cO_{F_\e}^\times$ is $E_F:=B_F \cap \cO_{F_\e}^\times$.\\ \\
We next define the group of Stark units, whose intersection with $\cO_F^\times$  is  of finite index in $\cO_F^\times$. For an irreducible complex character $\vartheta$ of $G$ and $s \in \mathbb{C}$ with $\Re(s)>1$, let
$$L_{F/k}(s,\vartheta):=\prod_{\mathfrak{b} \notin P_F}{(1-\vartheta((F/k, \mathfrak{b})) (\mathcal{N}(\mathfrak{b}))^{-s})^{-1}}$$ be the Artin $L$-function attached to $\vartheta$. The product is over all places in $k$ excluding $P_F$, the set of places that ramify in $F/k$.\\ \\
Extend $\vartheta$ linearly to $\mathbb{C}[G]$. The Stickelberger element $\Theta_F$ is the unique element in $\mathbb{C}[G]$ such that for all non trivial irreducible complex character $\vartheta$ of $G$,  $$\vartheta(\Theta_F) = (q-1) L_{F/k}(0,\bar{\vartheta}),$$ where  $\bar{\vartheta}$ is the complex conjugate of $\vartheta$. From the proof of the  Brumer-Stark conjecture over function fields \cite[Chapter V]{del}\cite{hay}, $\Theta_F\in \Z[G]$ and $\Theta_F$ annihilates $\cl_F^0$. Thus, for a divisor $\mathfrak{D}$ of degree $0$ in $F$, $\exists\ \alpha_\mathfrak{D} \in F$ uniquely determined up to a root of unity such that $\Theta_F(\mathfrak{D}) =[\alpha_\mathfrak{D}]_F$. The following stronger claim is proven in \cite{del}\cite{hay}[Theorem 1.1].\\ \\
Let $\mathfrak{R}$ be a prime divisor in $F$. There exists $\alpha_{\mathfrak{R}} \in F$ unique up to a root of unity such that  
\begin{enumerate}
\item $\Theta_F \in \Z[G]$
\item If $P_F$ is of cardinality greater than $1$, then $\Theta_F (\mathfrak{R}) = [\alpha_{\mathfrak{R}}]_F$.\\
If $P_F = \{\mathfrak{b}\}$ for a place $\mathfrak{b}$ in $k$, then $\Theta_F(\mathfrak{R}) + \mathfrak{B}  = [\alpha_{\mathfrak{R}}]_F$ where $\mathfrak{B}$ is the sum of places in $F$ above $\mathfrak{b}$.
\item $F(\alpha_{\mathfrak{R}}^{1/(q-1)})/k$ is abelian.
\end{enumerate}
Let $J_{F} \subset \Z[G]$ denote the annihilator of the roots of unity in $F$. Since $F(\alpha_{\mathfrak{R}}^{1/(q-1)})/k$ is abelian, from the characterization of $J_F$ in \cite{hay}[Lemma 2.5], it follows that for an $\eta \in J_F$, there exists $\lambda(\mathfrak{R},\eta) \in F$ unique up to a root of unity such that $\lambda(\mathfrak{R},\eta)^{q-1} = \alpha_{\mathfrak{R}}^{\eta}$.\\ \\
The subgroup $S_F$ of $F^\times$ generated by $\F_q^\times$ and  $\lambda(\mathfrak{R},\eta)$ as $\mathfrak{R}$ ranges over prime divisors in $F$ dividing $\infty$ and $\eta$ ranges over $J_F$ is defined as the group of Stark units. The Stark units are thus supported on places above $\infty$ and $\m$, and are thus $P_F$-units.\\ \\
Hayes \cite{hay} gave an explicit description of the $\alpha_{\mathfrak{R}}$ in terms of the $\m$-torsion points $\Lambda_\rho(\m)$. In particular, $\lambda_{\m}^{q-1}$ is a Stark unit in $H_\m$ \cite{hay}[\S~ 4,6] and $S_F$ is generated by $\mu(F)$ and $N_{H_\f/F_\f}(\lambda_\f^{(N(\g)- (\g,k_m/k))})$ \cite{ouk}[\S~ 3] where $\f \subset \cO_k$ ranges over non zero ideals and $\g \subset \cO_k$ ranges over non trivial ideals coprime to $\f$.
\section{Algorithms for Computing the Divisor Class Groups}\label{comp}
\noindent In this section, we develop the algorithms for the proofs of Theorem~\ref{ray_class_theorem}, Theorem \ref{stickelberger_generating_set} and Theorem~\ref{divisor_class_structure}. 
In this section, we assume that the point at infinity $\infty$ in $k$ chosen is of degree $1$. The index theorems concerning the Stickelberger ideal are known to hold only under this assumption. The rest of the algorithms, in particular those for determining the $\chi$ part of the ideal class number and ideal class group do not require this assumption.\\ \\
Further, for the following subsection wherein the computation of $\lambda_{\f,F}$ for $\f \mid \m$ is addressed, we restrict ourselves to the case where $k$ is $\F_q(t)$. 
\subsection{Computation of the narrow ray class field and Stark units}\label{compute_lambda}
For this subsection, we set $k$ to be the rational function field $\F_q(t)$. Since $\infty$ is of degree $1$, without loss of generality we may assume that $\cO_k = \F_q[t]$. For otherwise, we can perform an appropriate change of variable from $t$ to $s$ to ensure $k= \F_q(s)$ and $\cO_k=\F_q[s]$.
\begin{lemma}\label{lambda_lemma} There is a deterministic algorithm that given $\F_q(t)$ and a generator $f(t) \in \F_q[t]$ of an ideal $\f \subset \F_q[t]$, computes the minimal polynomial of $\lambda_\f$ over $\F_q(t)$ in time polynomial in $q^{\deg(\f)}$.
\end{lemma}
\begin{proof} Let $\f=\q_1^{c_1} \q_2^{c_2} \ldots \q_n^{c_n}$ be the factorization into powers of prime ideals. Using Berlekamp's algorithm \cite{ber}, we can obtain such a factorization deterministically in time polynomial in $\deg(\f)$ and $q$. From the factorization, obtain for each $i \in \{1,2,\ldots,n\}$, a monic irreducible $q_i(t) \in \F_q[t]$ such that $\q_i^{c_i} = (q_i^{c_i}(t))$.\\ \\
We first describe the computation of $\lambda_{\q_i^{c_i}}$ based on standard results found in \cite{hay_exp}\cite{ros}[Chap 12]. 
For $k=\F_q(t)$ and $\cO_k = \F_q[t]$, $\rho$ is the ring homomorphism that maps $t$ to $\rho_t = t + \tau_q$ and thus $\rho_t(y) = t + y^q$. Further, the minimal polynomial of $\lambda_{\q_i^{c_i}}$ over $\F_q(t)$ is $$\rho_{q_i^{c_i}(t)}(y)/\rho_{q_i^{c_i-1}(t)}(y).$$ Since $\rho_{q_i^{c_i}(t)}$ and $\rho_{q_i^{c_i}(t)}$ have degrees $c_i \deg(\q_i)$ and $(c_i-1) \deg(\q)$ in $\tau_q$ respectively, $\rho_{q_i^{c_i}(t)}(y)$ is of degree $q^{c_i \deg(\q_i)}$ and $\rho_{q_i^{c_i}(t)}(y)$ is of degree $q^{(c_i-1) \deg(\q)}$. Thus in time polynomial in $q^{c_i \deg(\q_i)}$ we can compute the minimal polynomial of $\lambda_{\q_i^{c_i}}$ over $\F_q(t)$.\\ \\ 
As an induction hypothesis, assume that $\lambda_{\ma}$ and $\lambda_{\mathfrak{b}}$ have been computed for all non trivial and relatively prime ideals $\ma$ and $\mathfrak{b}$ such that $\f = \ma \mathfrak{b}$.\\ \\
Let $a(t)$ and $b(t)$ respectively generate $\ma$ and $\mathfrak{b}$.
Using the extended Euclidean algorithm over $\F_q[t]$, compute $c(t),d(t) \in \F_q[t]$ such that $1 = c(t)a(t)+d(t)b(t)$. Since $\rho$ is a ring homomorphism, $\rho_1 = \rho_{c(t)}\rho_{a(t)}+ \rho_{d(t)}\rho_{b(t)}$ and its application on an $\F_q[t]$ module generator $\lambda$ of $\Lambda_\f$ yields $$ \lambda =  \rho_{c(t)}(\rho_{a(t)}(\lambda))+ \rho_{d(t)}(\rho_{b(t)}(\lambda)).$$
Since the $\F_q[t]$ submodules of $\Lambda_\f$ generated by $\rho_{a(t)}(\lambda)$ and $\rho_{b(t)}(\lambda)$ are respectively $\Lambda_{\mathfrak{b}}$ and $\Lambda_{\ma}$, it follows that there exist $\F_q[t]$ module generators $\hat{\lambda}_{\ma}$ and $\hat{\lambda}_{\mathfrak{b}}$ of $\Lambda_{\ma}$ and $\Lambda_{\mathfrak{b}}$ respectively such that $\rho_{c(t)}(\hat{\lambda}_{\mathfrak{b}})+ \rho_{d(t)}(\hat{\lambda}_{\ma})$ generates $\Lambda_\f$ as an $\F_q[t]$ module. \\ \\
Given an element in $\Lambda_\f$, since we know the factorization of $\f$, we can efficiently test if it generates $\Lambda_\f$ as an $\F_q[t]$ module by testing if a proper factor of $\f$ annihilates it. Since there are $|(\F_q[t]/\ma)^\times||(\F_q[t]/\mathfrak{b})^\times| \leq q^{\deg \f}$ choices for $(\hat{\lambda}_{\ma},\hat{\lambda}_{\mathfrak{b}})$ we may try them all to find a $\F_q[t]$ module generator $\lambda_\f$ of $\Lambda_\f$.
\end{proof}
\noindent As a corollary, given a generator of an ideal $\f \subset \cO_k$, we can construct $K_\f$ as $\F_q(t)(\lambda_\f)$ and $H_\f$ as $\F_q(t)(-\lambda_\f^{q-1})$ in time polynomial in $q^{\deg(\f)}$.\\ \\
For a totally real finite abelian extension $F$ presented as an irreducible polynomial $X_F(y) \in \F_q(t)[y]$ such that $F = \F_q(t)[y]/(X_F(y))$ along with a generator for the conductor $\m$ of $F$, we can explicitly compute the inclusion $F \hookrightarrow H_\m$ as follows. Factor $X_{F}(y)$ in $H_\m[y]$ and from the resulting splitting express a root of $X_F(y)$ as a polynomial in $-\lambda_\m^{q-1}$ with $\F_q(t)$ coefficients. The factorization takes time polynomial in $q$ and $[H_\m:k]$ and the size of $X_F$ \cite{pos}. Thus the total running time for computing the inclusion is bounded by a polynomial in $q^{\deg(\m)}$. As a consequence for an ideal $\f$ dividing $\m$, by considering $H_\f$ and $F$ as subfields of $H_\m$, we can construct $F_\f = H_\f \cap F$ in time polynomial in $q^{\deg(\m)}$ and the size of $X_F$.
\begin{lemma}\label{lambdaF_lemma} There is a deterministic algorithm that given $\F_q(t)$, a totally real finite abelian extension $F/\F_q(t)$ presented as an irreducible polynomial $X_F(y) \in \F_q(t)[y]$ such that $F = \F_q(t)[y]/(X_F(y))$ and a generator of the conductor $\m$ of $F$ finds $\lambda_{\f,F}$ for all $\f$ dividing $\m$ in time polynomial in $q^{\deg(\m)}$ and the size of $X_F$.
\end{lemma}
\begin{proof} We first obtain the factorization of $\m$ into prime power ideals in $\F_q[t]$ using Berlekamp's deterministic polynomial factorization algorithm over finite fields \cite{ber}.\\ \\
For each $\f$ dividing $\m$, let $X_{\f,F}(y) \in F_\f[y]$ denote the minimal polynomial of $-\lambda_\f^{q-1}$ over $F_f$ and let $X_\f(y) \in k[y]$ denote the minimal polynomial of $-\lambda_\f^{q-1}$ over $k$.\\ \\
For each $\f$ dividing $\m$, factor $X_\f(y)$ over $F_\f$ thereby obtaining the factorization
$$X_\f(y)  = \prod_{\theta \in Gal(F_\f/\F_q(t))} (X_{\f,F}(y))^\theta$$ where $(X_{\f,F}(y))^\theta$ denotes $X_{\f,F}$ with its coefficients acted on by $\theta$. We can read off $\lambda_{\f,F} = N_{H_\f/F_\f}(-\lambda_\f^{q-1})$ as $X_{\f,F}(0)$ up to a $Gal(F_\f/\F_q(t))$ conjugate. 
\end{proof}
\subsection{Computation of the Conductor}\label{conductor}
In this subsection, we sketch how to compute the conductor of $F$ given an irreducible polynomial $X_F(y) \in \F_q(t)[y]$ such that $F = \F_q(t)[y]/(X_F(y))$.\\ \\
First we compute the discriminant $\mathfrak{d}(X_F) \subseteq F_q[t]$ of the minimal polynomial $X_F$ and then find the set of all prime ideals in $\F_q[t]$ dividing $\mathfrak{d}(X_F)$. The running time is polynomial in $q$ and the size of $X_F$. Since $\m$ divides the discriminant of $F$ which in turn divides $\mathfrak{d}(X_F)$, the set of prime ideals dividing $\mathfrak{d}(X_F)$ contains the set of prime ideals dividing $\m$. To compute $\m$, for each prime $\q$ dividing $\mathfrak{d}(X_F)$, we have to determine the highest non negative integer $e$ such that $\q^e$ divides $\m$.\\ \\ 
Fix a prime ideal $\q$ dividing $\mathfrak{d}(X_F)$ and let $e$ denote the highest non negative integer such that $\q^e$ divides $\m$.\\ \\
Since $\q^e$ and $\m/\q^e$ are relatively prime, $H_\m$ is the composite  $H_{\m/\q^e} H_{\q^e}$ \cite{hay_exp}.\\ \\
Let $F_\q$ denote the localization of $F$ at a prime above $\q$ and let $k_\q$ denote the localization of $k$ at $\q$. For a positive integer $i$, let $H_{\q,i}$ denote the localization of $H_{\q^i}$ at a prime above $\q$ and let $k_\q^i$ denote the unique unramified extension of $k_\q$ of degree $i$.\\ \\
Since $F \subseteq H_\m = H_{\m/\q^e} H_{\q^e}$ and the localization of $H_{\m/\q^e}$ at a prime above $\q$ is unramified and of degree at most $[H_{\m/\q^e}:k]$, there exists a positive integer $i \leq [H_{\m/\q^e}:k]$ such that $F_\q \subseteq k_\q^i H_{\q,e}$. Denote by $n$ the smallest positive integer $i$ such that $F_\q \subseteq k_\q^i H_{\q,e}$. Further, there exists a positive integer $i$ such that $F_\q \subset k_\q^i H_{\q,j}$ if and only if $j \geq e$.\\ \\ 
This leads to the following algorithm to determine $e$. Start with $i=0$. If $F_\q \subseteq k_\q^i H_{\q,i}$ then output $i$ and terminate. Otherwise increment $i$.\\ \\
The algorithm terminates with the output $i_{out}$ that is at least $e$ and at most $\max(n,e)$. Given $i_{out}$, we can determine $e$ as the smallest non negative integer $j$ such that $F_\q \subseteq k_\q^{i_{out}} H_{\q,j}$. Since $\max(n,e)$ is bounded by $q^{\deg(\m)}$, we can determine $e$ with the number of trials bounded by a polynomial in $q^{\deg(\m)}$.\\ \\
\subsection{Computation of the Divisor Class Number: Proof of Theorem \ref{ray_class_theorem}}\label{divisor_class_number}
For an irreducible complex character $\vartheta$ of $G$ and $s \in \mathbb{C}$ with $\Re(s)>1$, let $$L_k(s,\vartheta):=\prod_{\mathfrak{b} \notin P_F(\vartheta)}{(1-\vartheta((F/k, \mathfrak{b})) (\mathcal{N}(\mathfrak{b}))^{-s})^{-1}}$$ be the Artin $L$-function attached to $\vartheta$. The product is over $P_F(\vartheta)$, the set of all places in $k$ not dividing the conductor of the character $\vartheta$. The analytic class number formula states that $$h(F) = h(k) \prod_{1 \neq \vartheta \in \hat{G}} L_k(0,\vartheta)$$ where the product is over the group of irreducible complex characters excluding the principal character.\\ \\
For $F$ of conductor $\m$, we next describe how the L-function evaluations $L_k(0,\vartheta)$ can be derived from the Stark unit $\lambda_{m,F}$.\\ \\
Associated with a Galois extension $L/k$ is the logarithm map $$\log_L : L \longrightarrow \Q[Gal(L/k)]$$ $$\ \ \ \ \ \ \ \ \ \ \ \ \ \ \ \ \ \ \ \ \ \ \ \ \ \ h \longmapsto \sum_{\sigma \in Gal(L/k)} \mathcal{V}_{\infty(L)}(h^\sigma)\sigma^{-1}.$$
Here $\mathcal{V}_{\infty(L)}$ denotes the valuation at a place $\infty(L)$ in $L$ above $\infty$. By \cite[\S~6]{hay}, the Stickelberger element $\Theta_{H_\f}$ can be computed from the image of $-\lambda_\f^{q-1}$ under the logarithm map as $$\Theta_{H_\f} = \frac{1}{q-1} \log_{H_\f}(-\lambda_\f^{q-1})$$
For a finte abelian extension $L/F$, by \cite[Lem 2.3] {ahn} $$\log_{F}(\N_{L/F}(z)) = Res_{L/F}(\log_{L}(z)), \forall z \in L.$$
Since the places that ramify in $H_\m/k$ and the places that ramify in $F/k$ have the same support, $\Theta_F = Res_{H_\m/F}(\Theta_{H_\m})$ and  $$\Theta_{F} = \frac{1}{q-1} \log_{F}(\lambda_{\m,F}).$$
For a non principal $\vartheta \in \hat{G}$ linearly extended to $\Z[G]$, $$L_{F/k}(0,\bar{\vartheta}) = \frac{1}{q-1} \vartheta(\Theta_F) = \vartheta \left(\log_{F}(\lambda_{\m,F})\right)$$
$$\Rightarrow L_{F/k}(0,\vartheta)  = \bar{\vartheta} \left(\log_{F}(\lambda_{\m,F})\right). $$
To compute the divisor class number $h(F)$ using the analytic class number formula, we require $L_k(0,\vartheta)$. But the $\bar{\vartheta}$ component of the Stark unit $\lambda_{m,F}$ under the logarithm map only yields $L_{F/k}(0,\vartheta)$. However, $L_k(0,\vartheta)$ and $L_{F/k}(0,\vartheta)$ are off only by finitely many Euler factors.\\ \\
In particular, the set of places that divide the conductor of $\vartheta$ is the set of places that are ramified in the extension $F/F_{\vartheta}$ where $F_{\vartheta}$ is the fixed field of $\vartheta$ in $F$. Thus $P_F \setminus P_F(\vartheta)$ consists of the places that ramify in $F/k$ but not in $F/F_{\vartheta}$. Hence for all non principal $\vartheta$, $$ L_k(0, \vartheta)  = L_{F/k}(0,\vartheta) \prod_{\mathfrak{b} \in P_F \setminus P_F(\vartheta)} (1-\vartheta(F/F_{\vartheta}, \mathfrak{b}))^{-1}$$
Hence the divisor class number can be derived from the Stark unit $\lambda_{\m,F}$ as $$h(F) = h(k) \prod_{1 \neq \vartheta \in \hat{G}}   \left( \bar{\vartheta} \left(\log_{F}(\lambda_{\m,F})\right) \prod_{\mathfrak{b} \in P_F \setminus P_F(\vartheta)} (1-\vartheta(F/F_{\vartheta}, \mathfrak{b}))^{-1} \right)$$
Given $\lambda_{\m,F}$, the running time of the algorithm to compute $h(F)$ is bounded by a polynomial in $q$ and the size of $X_F$. From lemma \ref{lambdaF_lemma}, $\lambda_{\m,F}$ can be computed in time polynomial in $q^{\deg(\m)}$ and theorem \ref{ray_class_theorem} follows.
\subsection{The Stickelberger Ideal: Proof of Theorem \ref{stickelberger_generating_set}}\label{stickelberger}
In \cite{yin}, Yin defined an ideal $I_F$ in $\Z[G]$ that annihilates $\cl_F^0$. In addition, an index formula was derived that shows that $[\Z[G]:I_{K_\m}]$ is up to a computable power of $(q-1)$ the divisor class number $h(K_\m)$. In \cite{ahn}, the index theorem was extended to hold for all abelian extensions $F/k$ that we consider.\\ \\
Following \cite{yin}, we define $Q_F$, a $G$ submodule of $\mathbb{Q}[G]$ such that $I_F= Q_F \cap \mathbb{Z}[G]$. The module $Q_F$ is comprised of a ramified part $Q_F^{ra}$ and an unramified part $Q_F^{ur}$.\\ \\
For an abelian extension $L/k$ and a subextension $\bar{L}/k$, let $Res_{L/\bar{L}}$ be the linear extension to $\Q[Gal(L/k)]$-modules of the restriction map from $Gal(L/k)$ to $Gal(\bar{L}/k)$. Likewise, let $Cor_{L/\bar{L}}$ be linear extension to $\Q[Gal(\bar{L}/k)]$-modules of the corestriction map from $Gal(\bar{L}/k)$ to $Gal(L/k)$.\\ \\
For $S \subseteq P_F$, let $F_S$ be the maximial subextension of $F$ where the primes outside $S$ are unramified. The ramified part $Q_F^{ra}$ is defined to be the $\Z[G]$-module generated by $\{Cor_{F/F_S}(\Theta_{F_S})\}_{S \subseteq P_F}$, the Stickelberger elements of the maximal subextensions $F_S$ under the conorm map \cite[\S~2]{yin}. For $F \subseteq H_\m$, $Q_F^{ra}$ is generated by $\{Cor_{F/F_\f^+} Res_{H_\f/F_\f^+} (\Theta_{H_\f}) \}_{\f\mid \m}$, where $\f$ ranges over all non trivial prime ideals dividing $\m$ \cite[\S~4]{ahn}.\\ \\
The unramified part $Q_F^{un}$ is the $\Z[G]$-module generated by $\{Cor_{F/F_\e} Res_{H_\g/F_\e}(\Theta_{H_\g}) \}_\g$ and $\frac{1}{q-1}\sum_{\sigma \in G} \sigma$ where $\g$ ranges over all prime ideals of $\cO_k$.\\ \\
Let $R$ be a finite set of ideals of $\cO_k$ such that $R$ is mapped surjectively onto $Gal(K_\e/k)$ under the Artin map which takes $\ma \in R$ to $(\ma,K_\e/k)$. We further require that for every $\ma \in R$, $(\ma,K_\e/k)$ is not the identity. Let $M=\{f \in k^\times: sgn(f)=1\}$ be the multiplicative group of positive functions in $k$. Let $P$ be the subgroup of $K_\e^\times$ generated by $\{\xi(\cO_k)/\xi(\g)\}_\g$ where $g$ ranges over all ideals of $\cO_k$. Let $\bar{P}$ be the subgroup of $P$ generated by $\{\xi(\cO_k)/\xi(\h)\}_{\h \in R}$. Hayes \cite[Equation 1.9]{hay_elliptic} proved that $P$ is the direct product of $\bar{P}$ and $M$. Further $P \cap k^\times=M$ \cite[Cor 2.5]{hay_elliptic}. Hence for every $\eta \in P$, there exists integers $b_\ma$ and a $\gamma \in M \subset k^\times$ such that $$\eta = \gamma \prod_{\ma \in R}\left( \xi(\cO_k)/\xi(\ma)\right)^{b_\ma}$$
Further, $\lambda_{\g,F} = \N_{K_\e/F_\e}(\xi(\cO_k)/\xi(\g))$ for all ideals $\g$ \cite{hay_elliptic}\cite[\S~3]{ahn}. For every ideal $\ma$ such that $(\ma,K_\e/k) \in Gal(K_\e/F_\e)$, $\N_{K_\e/F_\e}(\xi(\cO_k)/\xi(\g)) \in M \subset k^\times$ \cite[Lem 3.4]{ahn}. \\ \\
For $\gamma \in M \subset k^\times$, $\log_{F_\g}(\gamma) = c \sum_{\sigma \in G} \sigma$ for some integer $c$. Hence for every ideal $\g$, there exists integers $c_{\ma}$ and $c_0$ such that $$\log_{F_\g}(\lambda_{\g,F}) = c_0 \sum_{\sigma \in G} \sigma+ \sum_{\ma \in R} c_{\ma} \log_{F_{\ma}}(\lambda_{\ma,F})$$
Hence $Q_F$ is generated as a $\Z[G]$-module by $\{Cor_{F/F_\e} Res_{F_\g/F_\e}(\log_{F_\g}(\lambda_{\g,F})) \}_{\g \in R}$, $\frac{1}{q-1}\sum_{\sigma \in G} \sigma$ and $\{Cor_{F/F_\f^+} Res_{F_\f/F_\f^+} (\log_{F_\f}(\lambda_{\f,F})) \}_{f \mid \m}$.\\ \\
Let $C:=\{\g \subset \cO_k; \deg(\g) \leq 2 \log_q([K_\e:k])\}$ be a set of prime ideals. From Chebotarev's density theorem \cite{mur}, $C$ is mapped surjectively on to $Gal(K_\e/k)$ under the Artin map. Compute $\lambda_{\f,F}$ for every $\f$ that either divides $\m$ or is in $C$ and write down a generating set $M_F$ for $Q_F$ as a $\Z[G]$-module. Let $\Theta = \frac{1}{q-1}\sum_{\sigma \in G} \sigma$. Further, $Q_F=I_F+\mathbb{Z}\Theta$ \cite{ahn}. For every $s \in M_F$, define $i_s:=s - z_s \Theta$ where $z_s$ is the unique integer such that $0 \leq z_s < q-1$ and $i_s \in I_F$. Clearly $Z_F:=\{i_s\}_{s \in M_F} \cup  \{(q-1)\Theta\}$ generates $I_F$ as a $\mathbb{Z}[G]$-module.\\ \\
For $k=\F_q(t)$, given $\lambda_{\f,F}$ for $\f$ dividing $\m$, the running time of the algorithm to compute $M_F$ is bounded by a polynomial in $q$ and the size of $X_F$. From lemma \ref{lambdaF_lemma}, $\lambda_{\f,F}$ for $\f$ dividing $\m$ can be computed in time polynomial in $q^{\deg(\m)}$ and theorem \ref{stickelberger_generating_set} follows.
\subsection{Stickelberger Ideal and Structure of the Divisor Class Group: Proof of Theorem \ref{divisor_class_structure}}\label{stickelberger_structure}
In this section, we show that if the divisor class group $\mathcal{C}l^0_F$ is a cyclic $\Z[G]$ submodule, then the structure of $\mathcal{C}l^0_F$ is determined by the Stickelberger ideal $I_F$ upto. This leads to an algorithm to compute the structure of $\cl^0_F$ as an abelian group. Further, given a $\Z[G]$ generator, we can compute the invariant decomposition of $\cl^0_F$.\\ \\   
The index $[\Z[G] : I_F]$ of the Stickelberger ideal was computed in \cite{ahn} as $h(F)/h(k)$ up to a factor that is supported over primes dividing $[F:k]$. A precise statement of their index theorem follows.\\ \\
Let $T_{\q,F} \mathrel{\leq} G$ be the inertia group of $\q$ and $\sigma_{\q,F}:= (\q,F/k)^{-1} \left(\sum_{\tau \in T_{\q,F}} \tau\right)/(|I_{\q,F}|)$. For an ideal $\f$, define $$\alpha_{\f,F} : = \left(\sum_{\sigma \in Gal(F/F_f)} \sigma\right) \prod_{\q \mid \f} (1-\sigma_{\q,F}).$$ Let $e^+_{F} = \sum_{\tau \in G} \tau$. Let $V_F$ be the $\Q[G]$ module generated by $\{\alpha_{\f,F}\}_{\f \mid \m}$ and let $U_F = V_F + (\sum_{\tau \in Gal(F/F_\e)}\tau) \Z[G]$. Ahn, Bae, Jung proved that \cite[Thm 4.11]{ahn} $$[\Z[G]:I_F] = \frac{h(F)[F_\e:k](\Z[G]: U_{F})}{h(k)}$$ 
Further, the set of prime divisors of $(\Z[G]: U_{F})$ is contained in the set of prime divisors of $[F:F_\e]$. Hence $$[\Z[G]:I_F] = \frac{h(F)B}{h(k)}$$ for some $B$ whose prime divisors are contained in the primes dividing $[F:k]$.\\ \\ 
Let $r$ be the largest factor of $[\Z[G]:I_F]$ that is relatively prime to $h(k)[F:k]$. From the index theorem, it follows that the largest factor of $h(F)$ that is relatively prime to $h(K)[F:k]$ is $r$. Let $s_1 = [\Z[G]:I_F]/r$ and $s_2 = h(F)/r$. For an abelian group $K$, let $K[n]$ denote its $n$-torsion. Then,
$$\cl^0_F = \cl^0_F[r] \oplus \cl^0_F [s_2].$$\\
Assume that $\cl^0_F[r]$ is a cyclic $\Z[G]$ module. If $\gamma$ generates $\cl^0_F[r]$ as a $\Z[G]$ module and $J \subseteq \Z[G]$ denotes the annihilator of $\gamma$, then $$\Z[G]/J \cong \cl^0_F[r] = \Z[G]\gamma.$$ 
Since $I_F$ annihilates $\cl^0_F$, $I_F \subseteq J$ and there is a natural surjection $$\Z[G]/I_F \twoheadrightarrow \Z[G]/J$$ which implies that there is a surjection $$\phi : (\Z[G]/I_F)[r] \oplus (\Z[G]/I_F)[s_1] \twoheadrightarrow \cl^0_F[r]. $$
Since $r$ and $s_1$ are coprime, $(\Z[G]/I_F)[s_1]$ is in the kernel of $\phi$ and the restriction of $\phi$ to $(\Z[G]/I_F)[r]$ is surjective.\\ \\
Since $\left|(\Z[G]/I_F)[r]\right| = \left|\cl^0_F[r]\right| = r,$ the restriction of $\phi$ to $(\Z[G]/I_F)[r]$ is a $\Z[G]$ module isomorphism.\\ \\
Since $(\Z[G]/I_F)[r]$ and $\cl^0_F[r]$ are isomorphic as $\Z[G]$ modules, they are isomorphic as groups.\\ \\
Recall that $Z_F$ is a finite set that generates $I_F$ as a $\Z[G]$ module. Thus $\bar{Z}_F : = \{\sigma(z) | \sigma \in G, z \in Z_F\}$ generates $I_F$ as a $\Z$-module and we can determine the structure of $(\Z[G]/I_F)[r]$. Further, given $\gamma$, we can compute the invariant factor decomposition of $\cl^0_F[r]$ by a Smith normal form computation. Further, the Smith normal form computation yields a unimodular projection matrix that allows us to efficiently project a divisor class written as $\sum_{\sigma \in G} a_\sigma \gamma^\sigma$ for $a_\sigma \in \Z$ into the the invariant factor decomposition of $\cl^0_F[r]$.\\ \\
We next turn our attention to $\cl^0_F[s_2]$.\\ \\
Let $\ell$ be a prime dividing $s_2$ and let $(\cl^0_F)_\ell$ be the $\ell$-primary component of $\cl^0_F$. Let $\ell^{b_\ell}$ be the exponent of $(\cl^0_F)_\ell$.\\ \\
If $A$ is a cyclic subgroup of $(\cl^0_F)_\ell$ of order $\ell^{b^\ell}$, then $(\cl^0_F)_\ell$ is the direct sum of $A$ and its complement. Thus, if we can find a cyclic subgroup A of $(\cl^0_F)_\ell$ of order $\ell^{b_\ell}$, then the problem of computing the structure of $(\cl^0_F)_\ell$ is reduced to computing the structure of the complement of $A$ and we can proceed inductively.\\ \\
There is a set $A_F$ of degree zero divisors of size polynomial in $[F:k]$ and $\log q$ whose divisor classes generate $\cl^0_F$ \cite{hess}[Theorem 34]. Further, $A_F$ consists of divisors of pole degree bounded by $\mathcal{O}(\log [F:k])$.\\ \\
Thus, $$(\cl^0_F)_\ell = \left\langle \left\{\frac{h(F)}{| (\cl^0_F)_\ell | } \bar{H},  H \in A \right\}\right\rangle$$ and
at least one of the elements in the above generating set of $(\cl^0_F)_\ell$ has order $\ell^{b_\ell}$. Further, given a degree zero divisor with pole degree $\delta$ and an integer $a$, we can test if it is principal using a Riemann-Roch computation in time polynomial in $[F:k]$, $\log q$, $\delta$ and $\log(\ell^a)$ \cite{hesrr}. Hence we can find the element of maximal order in the generating set and the subgroup generated by it would be the $A$ that we seek.\\ \\
By computing the structure of $(\cl^0_F)_\ell$ for every $\ell$ dividing $s$, we determine $(\cl^0_F)[s_2]$.\\ \\
Thus we can obtain the invariant decomposition $\cl^0_F$ of the form $$\mathcal{C}l^0_F= \langle e(1) \gamma \rangle \oplus \langle e(2)\gamma \rangle \oplus \ldots \oplus \langle e([F:k])\gamma \rangle $$
where for $1\leq i \leq [F:k]$, $e(i) \in \Z[G]$ and $d_i$ the order of $e(i)\gamma$ in $\cl^0_F$ and for $1 \leq i < [F:k]$, $d_i \mid d_{i+1}$. Thus theorem \ref{divisor_class_structure} follows.\\ \\
Given two degree zero divisors $D_1,D_2$, the discrete logarithm problem in $\mathcal{C}l^0_F$ is to compute an integer $x$ such that $\bar{D_1} \sim x \bar{D_2}$ if it exists. The discrete logarithm problem over $\mathcal{C}l^0_F$ is believed to be hard. There are several cryptosystems whose security is reliant on the hardness of solving the discrete logarithm problem, in particular when $F$ is the function field of an elliptic curve.\\ \\
Assume that $D$ is a degree zero divisor that generates $\cl^0_F$ and that $D$ has $[F:k]$ distinct conjugates. The above decomposition allows us to project an degree zero divisor in $\Z[G]D$ in to the invariant decomposition of $\cl^0_F$. This reduces the discrete logarithm problem between two divisors in $\Z[G]D$ to inversion in $\Z/d_1\Z \oplus \Z/d_2\Z \oplus \ldots \oplus \Z/d_{r}\Z$ which can be solved efficiently using the extended Euclidean algorithm.
\section{Euler Systems from Stark Units}\label{euler}
Let $\ell$ be a prime number not dividing $q(q^{d_\infty}-1)[F:k]$ and  $N$ a power of $\ell$. Let $\ell^a$ be the cardinality of $\p_\ell(\cO_F)$. Fix a finite set $\{\h_1,\h_2,\ldots,\h_s\}$ of ideals of $\cO_k$ such that $\p_\ell(\cO_k)$, the $\ell$-primary part of $\p(\cO_k)$ decomposes as
$$\p_\ell(\cO_k)= \langle \hat{\h_1}\rangle \times \langle \hat{\h_2}\rangle \times \ldots \times \langle \hat{\h_s}\rangle$$
where for $1\leq i\leq s$, $\hat{\h_i}$ is the class of $\h_i$ in $\p(\cO_k)$. For $1\leq i \leq s$, let $n_i$ be the order of $\langle \hat{\h_1}\rangle$. Fix a $h_i \in k$ such that $(\h_i)^{n_i}=h_i\cO_k$.
Let $R_N$ be the set of prime ideals of $\cO_k$ that split completely in the extension $F^\prime:=F(\mu_N,h_1^{1/N},h_2^{1/N},\ldots,h_s^{1/N})/k$. For every $\pe \in R_N$, there exists a cyclic extension $F(\pe)/F$ of degree $[F(\pe):F]=N$ such that $F(\pe) \subset FH_\pe$, $F(\pe)/F$ is unramified outside $\pe$ and the primes in $F$ above $\pe$ are totally ramified in $F(\pe)/F$ \cite[Lem 3.1]{ouk}. Fix a $\sigma_\pe$ such that $\langle \sigma_\pe \rangle = Gal(F(\pe)/F)$.\\ \\
Let $B_N$ be the set of square free products of ideals in $R_N$. For an $\ma = \pe_1\pe_2 \ldots \pe_b \in B_N$ with $\pe_1,\pe_2,\ldots ,\pe_b \in R_N$, let $F(\ma)$ denote the compositum $F(\pe_1)F(\pe_2)\ldots F(\pe_b)$. For the unit ideal $\e= \cO_k$, let $F(\e):=F$.\\ \\
An Euler system of modulus $N$ is a function $\Psi : B_N \longrightarrow k_\infty^\times$ such that $\forall \ma \in B_N$ and $\forall \pe \in R_N$,
\begin{enumerate}
\item  $\Psi(\ma) \in F(\ma)^\times$
\item  If $\ma \neq \e$, then $\Psi(\ma) \in \cO_{F(\ma)}^\times$
\item  $\N_{F(\ma\pe)/F(\ma)}(\Psi(\ma\pe))= \Psi(\ma)^{1-(\pe,F(\ma)/k)^{-1}}$	
\item  $\Psi(\ma\pe)=\Psi(\ma)^{(\pe,F(\ma)/k)^{-1} (\N(\pe)-1)/N}$ modulo every prime in $F(\ma\pe)$ above $\pe$.
\end{enumerate}
Oukhaba and Vigue \cite[\S~3]{ouk} proved that for every non zero coprime ideals $\f,\g  \subset \cO_k$, $$\Psi_{\f,\g}(\ma):=\mathcal{N}_{H_{\f\ma}/F(\ma)}(\lambda_{\f\ma}^{\mathcal{N}(\g)-(\g,K_{\f\ma}/k)})$$
is an Euler system such that $\Psi_{\f,\g}(\e)=\mathcal{N}_{H_\f/H_\f \cap F}(\lambda_\m^{\mathcal{N}(\g)-(\g,K_\f/k)})$.\\ \\
Since $\{\mathcal{N}_{H_\f/H_\f \cap F}(\lambda_\m^{\mathcal{N}(\g)-(\g,K_\m/k)})\}_{\f,\g}$ generates $S_F$ up to roots of unity and the product of two Euler systems is an Euler system, for every $\alpha \in S_F$, there exists an Euler system $\Psi$ such that $\Psi (\e)=\alpha$ \cite[Cor 3.16]{ouk}. If $\Psi (\e) =\alpha$, we call $\Psi$ an Euler system starting from $\alpha$.
\subsection{Kolyvagin Systems of Derivative Classes}\label{koly}
From an Euler system $\Psi$, a collection of functions $\ka(\ma) \in F^\times$ indexed by $\ma \in B_N$ is derived. The places that appear in the divisor $[\ka(\ma)]$ admit a precise characterization up to an $N^{th}$ multiple due to the properties of Euler systems.\\ \\
For $\pe \in R_N$, let $D_\pe:= \sum_{i=0}^{N-1}i\sigma_\pe$. For an $\ma \in B_N$, let $D_\ma:=\prod_{\pe/\ma}D_\pe$ where the product is over prime ideals $\pe$ dividing $\ma$.\\ \\
For every $\sigma \in Gal(F(\ma)/F)$ and every prime $\pe$ dividing $\ma$, the class of $\Psi(\pe)^{(\sigma-1)D_\ma}$ in $F(\ma)^\times/(F(\ma)^\times)^N$ is fixed by $Gal(F(\ma)/F)$ \cite[Lem~4.1]{ouk}.\\ \\
The $N^{th}$ roots of unity are trivial in $F(\ma)$ and the $1$-cocycle $C_{\ma}:Gal(F(\ma)/F) \longrightarrow F(\ma)^\times$ that takes $\sigma$ to $\Psi(\ma)^{(\sigma-1)D_\ma}$ is well defined.
Hilbert's theorem 90 implies that there exists a $\beta \in F(\ma)^\times$ such that $C_{\ma}(\sigma)=\beta^{\sigma-1}$ for all $\sigma \in Gal(F(\ma)/F)$. Set $\ka(\ma) := \frac{\Psi(\ma)^{D_\ma}}{\beta^N}$.
In particular, set \begin{equation}\label{kolyvagin_equation}\frac{1}{\beta}:=\sum_{\sigma \in Gal(F(\ma)/F)} C_{\ma}(\sigma)\sigma(e)=\sum_{\sigma \in Gal(F(\ma)/F)}\left(\Psi(\ma)^{(\sigma-1)D_\ma}\right)^{\frac{1}{N}} \sigma(e)\end{equation}
Here $e \in F(\ma)^\times$ is picked such that the term on the right does not vanish. Independence of characters implies the existence of such an $e$. For instance, $e=\lambda_{\ma,F(\ma)}$ assures that the term does not vanish. Set $e=\lambda_{\ma,F(\ma)}$ and $\ka(\ma) := \frac{\Psi(\ma)^{D_\ma}}{\beta^N}$.\\ \\
For every $\sigma \in Gal(F(\ma)/F)$, $$\left(\beta^{(\sigma-1)}\right)^N=\Psi(\ma)^{(\sigma-1)D\ma} \Rightarrow \left( \frac{\Psi(\ma)^{D_\ma}}{\beta^N}\right)^{(\sigma-1)}=1 \Rightarrow \ka(\ma) \in F^\times$$
Further, $\ka(\ma) = \Psi(\ma)^{D_\ma}$ modulo $(F^\times)^N$.\\ \\
Let $I$ be the group of fractional ideals of $\cO_F$ written additively as a subgroup of the group of divisors of $F$. For a prime ideal $\pe$ of $\cO_k$, let $I_\pe$ be the subgroup of $I$ supported at places in $F$ above $\pe$. For $f\in F^\times$, let $[f]_\pe \in I_\pe$ be the projection of $f\cO_F$ in $I_\pe$.\\ \\
From \cite[Prop4.3]{ouk}, for every $\pe \in R_N$, there exists a $G$-equivariant map
\[ \varphi_\pe : (\cO_F/\pe)^\times/((\cO_F/\pe)^\times)^N \rightarrow I_\pe/N I_\pe\]
unique up to a multiple of $(\Z/N\Z)^*$, that makes the following diagram commute.\\ \\
\begin{tikzpicture}[description/.style={fill=white,inner sep=2pt}]
\matrix (m) [matrix of math nodes, row sep=3em,
column sep=1.5em, text height=1.5ex, text depth=0.25ex]
{ & & & & & {F(\pe)^\times} & & & & &\\
 & & & {(\cO_F/\pe)^\times/((\cO_F/\pe)^\times)^N} & & & & {I_\pe/N I_\pe} & & &\\ };
\path[->,font=\scriptsize]
(m-1-6) edge node[description] {$x\rightarrow (x^{(1-\sigma_\pe)})^{\frac{1}{d}} $} (m-2-4)
(m-1-6) edge node[description] {$x\rightarrow [\N_{F(\pe)/F}(x)]_\pe $} (m-2-8)
(m-2-4) edge node[auto] {$\varphi_\pe$} (m-2-8);
\end{tikzpicture}\\
where $d = \frac{\N(\pe) -1 }{N}$.
Let $\pi \in \mathfrak{P} \setminus \mathfrak{P}^2$ where $\mathfrak{P}$ is a prime ideal in $F(\pe)$ above $\pe$. Let $\mathfrak{B} : = \mathfrak{P} \cap \mathcal{O}_F$.\\ \\
 Let $Gal(F(\pe)/F) = \langle \sigma_{\pe} \rangle$.  Then the image of  $\pi^{1-\sigma_{\pe}}$ is of order $N$ in
 $(\mathcal{O}_{F(\pe)}/\mathfrak{P})^\times \cong (\mathcal{O}_{F}/\mathfrak{B})^\times$, and it is independent of the choice of $\pi$.  We will denote this image by $\bar{\pi}_{\mathfrak{B}}$.\\ \\
The unique $G$-equivariant map $\varphi_{\pe}$ that makes the above diagram commute takes $$\bigoplus_{\mathfrak{B} \mid \pe} \bar{\pi}_{\mathfrak{B}}^{b_{\mathfrak{B}}} \longmapsto \sum_{\mathfrak{B} \mid \pe} b_{\mathfrak{B}} \mathfrak{B}.$$

\ \\The following lemma relates Kolyvagin derivative classes through the $\varphi_{\pe}$ map.
\begin{lemma}\label{kolyvagin_projection} (\cite[Lem~4.4]{ouk}) For $\pe \in B_N$, if $\pe \nmid \ma$ then $[\ka(\ma)]_\pe = 0 \mod N$ and if $\pe \mid \ma$ then $[\ka(\ma)]_\pe = \varphi_{\pe} (\ka(\ma/\pe)) \mod N$.
\end{lemma}
\section{Characterizing Ideal Class Group Using Kolyvagin Systems}\label{koly_class}
Let $\chi$ be a non trivial irreducible $\Z_{\ell}$-representation of $G$ of dimension $\dim(\chi)$ and $e(\chi) = \frac{1}{[F:k]}\sum_{\sigma\in G} \mbox{Tr} (\chi(\sigma)) \sigma^{-1}$ the corresponding idempotent in $\mathbb{Z}_l[G]$.  For a $\Z_{\ell} [G]$ module $B$, define $B(\chi) : = e(\chi)B$.  Let $U=\cO_F^\times$ and $E=S_F \cap \cO_F^\times$. Gras conjecture, proven true by Oukhaba and Viguie in this context, relates the cardinalities of $U/E$ and $\p(\cO_F)$.
\begin{theorem}(Gras Conjecture \cite[Thm 1.1]{ouk})
For every prime $\ell$ not dividing $q(q^{d_\infty}-1) [F : k]$, $|\p_\ell(\cO_F)(\chi)| = | (U/E) (\chi) |$ for every non trivial irreducible $\Z_{\ell}$ representation $\chi$ of $G$.
\end{theorem}
Theorem 1.1 in \cite{ouk} is stronger than what is stated here. It allows $\ell$ to divide $q^{d_\infty}-1$ as long as $\ell$ does not divide $[H_\m:k]$ and $\chi$ is of a certain form.\\ \\
More is known regarding the structure of $(U/E)(\chi)$. Since $e(\chi)\Z_{\ell} [G]$ is isomorphic to the ring of integers of the unramified abelian extension of $\Q_{\ell}$ of degree $\dim(\chi)$, $e(\chi)\Z_{\ell} [G]$ is a discrete valuation ring and every simple torsion $e(\chi)\Z_{\ell} [G]$-module is isomorphic to $\Z/\ell^c\Z [G] e(\chi)$ for some $c$. It is proven in \cite[Thm 4.8]{ouk} that $(U/E) (\chi)$ is $G$-isomorphic to $e(\chi) \Z/ t \Z [G]$ for some $t$ (which is power of $\ell$) such that $t (U/E)(\chi) =0$.
\subsection{Structure of the $\ell$-part of the Class Group: Proof of Theorem~\ref{chi-structure}}\label{structure}
We present a characterization of the structure of $\p_\ell(\cO_F)(\chi)$ in terms of Kolyvagin's derivative classes.
\ \\Let $N$ be a power of $\ell$.
For an $\alpha\in F^{\times}/F^{\times N}$, if $e(\chi) \alpha$ is an $\ell^c$-th power but not an $\ell^{c+1}$-th power in $F^\times/F^{\times N}(\chi)$, we define $\ell^c$ to be the $\chi$-index of $\alpha$.\\ \\
Since $U(\chi)$ modulo the roots of unity is a free rank one $e(\chi)\Z_{\ell} [G]$-module \cite[\S~4]{ouk}, there exists a $\lambda \in U$ whose projection in $e(\chi) (U/U^N)$ has order $N$. Thus $\lambda^{t}\in e(\chi) (E/U^N)$ and $\lambda^{t}$ has $\chi$-index $t$. Hence there exists elements in $E$ of $\chi$-index $t$. Consider a Kolyvagin system of modulus $N$ starting from a unit $\kappa(\e) \in E$ of $\chi$ index $t$.\\ \\
We introduce a concept that will be useful for extending the reasoning in the proof of Gras Conjecture in \cite{fen}\cite{ouk} to obtain further results.\\ \\
Write $\ka(\ma) \stackrel{\chi}{\to} \ka(\ma\pe)$, if there is a prime $\mathfrak{P} | \pe$ in $F$ and a
$u\in ((\Z/N\Z [G])(\chi))^{\times}$
such that $$u T \mathfrak{P} = [e(\chi) \ka(\ma\pe)]_\pe \mod N $$ where $T$ is the $\chi$-index of $\ka(\ma)$. If more specific, we write $\ka(\ma) \stackrel{\chi}{\to} \ka(\ma\pe)$ through $\mathfrak{P}$.\\ \\
Let $ord_N(q)$ be the order of $q$ in $(\Z/N\Z)^\times$. The following lemma without the requirement that $\mathfrak{P}$ is of degree at most $\max\{ord_N(q), 2\log_q(\ell^{4a+2}[F:k]) \}$ is proven in \cite[Thm 4.7]{ouk}. However, our computation needs the effective version stated below with the degree of $\mathfrak{P}$ bounded.
\begin{lemma}
\label{density}
Let $A$ be a $\Z_{\ell} [G]$-quotient of $\p(\cO_F)_\ell(\chi)$.
Let $H$ be the abelian extension of $F$ corresponding to $A$.
Let $F_N = F(\mu_N)$,  $F^\prime=F(\mu_N,h_1^{1/N},h_2^{1/N},\ldots,h_s^{1/N})$ and $L= F^\prime (W^{1/N})\cap H$.
Let $\beta\in (F^{\times}/F^{\times N} )(\chi)$ and $b$ be the order of of $\beta$ in $F^\times/F^{\times N}$.
Let  $W$ be a finite cyclic $G$-submodule of $F^\times/F^{\times N}$ generated by $\beta$.
Let $s$ be the number of factors in the primary decomposition of $A$.
Then there exists a $\Z[G]$- generator $c'$ of $Gal (L/F)$  such that for every $c\in A$ whose restriction to $L$ is $c'$, there exists a prime $\mathfrak{P}$ of $F$ of degree at most $\max\{ord_N(q), 2\log_q([F:k]),2(s+3)\log_q(N)\}$ such that
\begin{enumerate}
\item  The projection of $\mathfrak{P}$ in $A$ is in $c$.
\item  $\pe \in R_N$, where $\pe = \mathfrak{P} \cap k$.
\item  $[\beta]_\pe=0$ and there exists $u \in ((\mathbb{Z}/N\mathbb{Z}[G])(\chi))^\times$ such that $\varphi_\pe(\beta)= u  (N/b)  \mathfrak{P}$.
\end{enumerate}
\end{lemma}
\begin{proof} Everything in the lemma is proven in \cite{ouk} [Thm 4.7] except the degree bound on $\mathfrak{P}$.
As argued in the proof in \cite{ouk} [Thm 4.7], there exists $\tau\in Gal (F' (W^{1/N} /F')$ that generates
$Gal (F' (W^{1/N} /F')$ over $\Z[Gal(F_N/k)]$.  The restriction $c'$ of $\tau$ to $L$ is a $\Z[G]$-generator of $Gal (L/F)\cong Gal (LF'/F')$.
Let $\theta\in Gal (H/F) = A$ correspond to $c$ where $c$ is an extension of $c'$ to $H$. Choose $\rho\in
Gal (HF' (W^{1/N})/F)$ such that
\begin{eqnarray*}
\rho|_H & = &  \theta\\
\rho |_{F'(W^{1/N})} & = & \tau
\end{eqnarray*}
The field of constants of $HF^\prime(W^{1/N})$ is $\F_{q^{ord_N(q)}}$. Let $m$ be a multiple of $ord_{N}(q)$. To ensure $F$ contains places of degree $m$, further assume that $m > \log_q([F:k])$. Let $E$ be the conjugacy class of $\rho$ in $Gal(HF^\prime(W^{1/N})/F)$. Let $N_m(E)$ denote the cardinality of $S_m(E):=\{\mathfrak{s} | \deg(\mathfrak{s})=m, (HF^\prime(W^{1/N})/F,\mathfrak{s}) \in E\}$, where $\mathfrak{s}$ denotes a place in $F$ that is unramified in $HF^\prime(W^{1/N})/F$. By the Chebotarev density theorem \cite{isi},
$$\left\vert N_m(E) -  \frac{|E| ord_N(q) q^m}{m [HF^\prime(W^{1/N}):F ] } \right\vert \leq 6.5\ D\ [HF^\prime(W^{1/N}):F ]\ q^{m/2}$$
where $D$ is the degree of the different of the extension $HF^\prime(W^{1/N})/F$.\\ \\
The different $D$ is bounded by the genus of $HF^\prime(W^{1/N})$ and $[HF^\prime(W^{1/N}):F] \leq N^{s+3}$, where $s$ is the number of factors in the primary decomposition of $A$. Pick $m$ to be the smallest multiple of $ord_N(q)$ such that $m > (s+3) \log_q(N)$ and $m > \log_q([F:k])$, then $N_m(E)$ is non zero. Further, $N_m(E)$ is at least $\frac{1}{N^{s+3}}$ fraction of the number of places in $F$ of degree $m$.\\ \\
Pick a place $\mathfrak{P}$ in $S_m(E)$.
The rest of the proof is exactly the same as \cite{ouk} [Thm 4.7].
\end{proof}

\ \\For $\ma \in B_N$, let $\C_\ma$ denote the subgroup of $\p(\cO_F)_\ell$ generated by primes dividing $\ma$ in $F$.
\begin{lemma}
\label{chi-link}
Let $\chi$ be an irreducible $\Z_{\ell}$ representation of $G$. Let $\ma \in B_N$ and suppose $\C_\ma (\chi)$ is a proper subgroup of $\p_\ell(\cO_F)(\chi)$.  Let $A$ be the $\Z_{\ell}[G]$-quotient  $\p_\ell(\cO_F)(\chi)/\C_\ma (\chi)$.
Then there is a prime $\mathfrak{P}$ of $F$ that projects to a nontrivial class $c\in A$ such that $\mathfrak{P}$ is over a prime $\pe \in B_N$ with $\pe \nmid \ma$ and $\ka(\ma) \stackrel{\chi}{\to} \ka(\ma\pe)$ through $\mathfrak{P}$.
\end{lemma}
\begin{proof}The lemma follows by applying Lemma~\ref{density} to the finite $G$-submodule $W$ of $F^\times/F^{\times N}$ generated by $\beta=e(\chi)\ka(\ma)$. \end{proof}

\begin{lemma}
\label{chi-div}
Suppose $\ka(\ma) \stackrel{\chi}{\to} \ka(\ma\pe)$ through $\mathfrak{P}$.  Let $T$ be the $\chi$-index of $\ka(\ma)$ and $B$ the $\chi$-index of $\ka(\ma\pe)$.  Then $B | T$.   If $(N/T) \p_\ell(\cO_F)(\chi) =0$, then the class of $\mathfrak{P}$ in $\p_\ell(\cO_F)(\chi)/\C_\ma(\chi)$ has order dividing $T/B$.
\end{lemma}
\begin{proof}  We have $e(\chi)\ka(\ma\pe) = (e(\chi)\alpha)^{B}$ in $e(\chi) (F^{\times}/F^{\times N})$ for some $\alpha \in F^\times$.  On the other hand, since $\ka(\ma) \stackrel{\chi}{\to} \ka(\ma\pe)$ through $\mathfrak{P}$, $u T e(\chi) \mathfrak{P} = [e(\chi) \ka(\ma\pe)]_{\pe}$.  So $B e(\chi)[\alpha]_\pe \equiv u T e(\chi) \mathfrak{P} \mod N$.  Therefore $B \mid T$. Since $[B e(\chi) \alpha] \equiv [ e(\chi) \ka(\ma\pe) ] \mod N$ and $(N/B) \p_\ell(\cO_F)(\chi) =0$, $(T/B) e(\chi) \mathfrak{P} \equiv 0 \mod \C_\ma (\chi)$ and the lemma follows. \end{proof}

\ \\Consider a $\chi$-path starting from $\ka(\e)$:
\[ \ka(\e) \stackrel{\chi}{\to} \ka(\pe_1) \stackrel{\chi}{\to} \ka(\pe_1\pe_2) \stackrel{\chi}{\to} \ldots \stackrel{\chi}{\to} \ka(\pe_1 \pe_2 \ldots \pe_n)\]
We say that the $\chi$-path is complete if the $\chi$-index of the last node $\ka(\pe_1 \pe_2 \ldots \pe_n)$ is $1$.\\ \\
From now on we assume that $N=\ell t^2$.\\ \\
Suppose in the $\chi$-path above, $\ka(\pe_1\pe_2 \ldots \pe_{i-1}) \stackrel{\chi}{\to} \ka(\pe_1 \pe_2\ldots \pe_i)$ through $\mathfrak{P}_i | \pe_i$.
Note that for all primes $\mathfrak{P}$ in $F$, $e(\chi)\mathfrak{P}^{\sigma}=\sigma e(\chi) \mathfrak{P}$.  Hence for all $1\leq i \leq n$, $$e(\chi)\Z_{\ell}[G] \langle \mathfrak{P}_1, \mathfrak{P}_2, \ldots, \mathfrak{P}_i \rangle  =  \C_{\pe_1\pe_2\ldots \pe_i} (\chi) $$
Let $\C_i(\chi) = \C_{\pe_1\pe_2\ldots \pe_i} (\chi)$ and let $d=\dim(\chi)$. Let $t_i$ be the $\chi$-index of $\ka(\pe_1\pe_2\ldots \pe_i)$.
From Lemma~\ref{chi-div}, we see that that $[ \C_1(\chi) : 1 ]$ divides $(t/t_1)^d$ and for $i > 1$, $[ \C_i(\chi) : \C_{i-1}(\chi)]$ divides $(t_{i-1}/t_i)^d$.  It follows that $[ \C_n(\chi) :1]$ divides $(t/t_n)^d$.
Suppose $\C_n(\chi)= \p_\ell(\cO_F)(\chi)$, then $[\C_n(\chi) : 1] = t^d$ by Gras Conjecture, hence we must have $t_n = 1$.\\ \\
Conversely, suppose $t_n =1$. Suppose for a contradiction that $\C_n (\chi) \neq \p_\ell(\cO_F)(\chi)$.  By Lemma~\ref{chi-link}, there exists a prime $\mathfrak{P}$ that projects to a non-trivial class  $c\in\p_\ell(\cO_F)(\chi)/\C_n (\chi)$, such that $\mathfrak{P}$ is over a prime $\pe \in B_N$ not dividing $\pe_1 \pe_2 \ldots \pe_n$ and $\ka(\pe_1 \pe_2 \ldots \pe_n) \stackrel{\chi}{\to} \ka(\pe_1 \pe_2 \ldots \pe_n \pe)$ through $\mathfrak{P}$. Since $(N/t_n) \p(\cO_F)_{\ell}(\chi) =0$, it follows from Lemma~\ref{chi-div} that the class of $\mathfrak{P}$ in $\p_\ell(\cO_F)(\chi)$ (which is $c$), has order modulo $\C_n(\chi)$ dividing $t_n=1$, hence is 1. We have a contradiction.  Hence $\C_n(\chi) = \p_\ell(\cO_F)(\chi)$ and for all $1\leq i \leq n$, $[ \C_i(\chi) : \C_{i-1}(\chi) ] = (t_i/t_{i-1})^d$.\\ \\
Suppose $t_n > 1$.   Then $\C_n(\chi) \neq \p_\ell(\cO_F)(\chi)$.  By Lemma~\ref{chi-link}, there exists prime $\mathfrak{P}$ that projects to a non-trivial class $C \in \p_\ell(\cO_F)(\chi)/\C_n(\chi)$ such that $\mathfrak{P}$ is over a prime $\pe\in B_N$ not dividing $\pe_1 \pe_2 \ldots \pe_n$ and  $\ka(\pe_1 \pe_2 \ldots \pe_n) \stackrel{\chi}{\to} \ka(\pe_1 \pe_2 \ldots \pe_n \pe)$ through $\mathfrak{P}$. In this fashion, we may extend the $\chi$-path until the $\chi$-index of the last element is one or equivalently we have the entire $\p_\ell(\cO_F)(\chi)$ constructed. We have thus proven Theorem~\ref{chi-structure}.

\ \\Since $\C_i(\chi)/\C_{i-1}(\chi)$ is $G$-cyclic of exponent $t_{i-1}/t_i$, it follows that $\C_n (\chi)$ is of exponent
dividing $\prod_{i=1}^n t_{i-1}/t_i  = t_0$, which is the exponent of $(U/E)(\chi)$.  Therefore we have the following
\begin{theorem}
\label{chi-exponent}
The exponent of $\p_\ell(\cO_F)(\chi)$ divides the exponent of $(U/E)(\chi)$.
\end{theorem}
\noindent Theorem~\ref{chi-structure} leads to an iterative procedure to compute $\p_\ell(\cO_F)(\chi)$, which will be discussed in \S~\ref{ideal_class_structure}.

\section{Computation of the Ideal Class group}

\noindent We explore the algorithmic implications of Gras conjecture and Theorem~\ref{chi-structure}.
\subsection{Constructing $\ell$-adic Representations}
 We begin by constructing the $\mathbb{Z}_{\ell}$ character $\chi$ of $G$. It is sufficient for our algorithm to compute the associated idempotent $e(\chi)$. Consider the primary decomposition $G = \langle\sigma_1\rangle \oplus \langle\sigma_2\rangle \oplus \ldots \oplus \langle\sigma_s\rangle$ where for $1\leq i \leq s$,  $\sigma_i$ has order $q_i^{b_i}$ in $G$, where $q_i$ is a prime. For each $1 \leq i \leq s$, we describe all irreducible $\Z_\ell$ representations $\chi_i$ of the group $\langle\sigma_i\rangle$. For every choice of irreducible representations $\{\chi_i\}_i$, the Kronecker product $\bigotimes_i \chi_i$ defines an irreducible $\Z_\ell$ representation on $G$ and every irreducible $\Z_\ell$ representation of $G$ can be obtained in this manner.\\ \\
Consider the factorization $x^{q_i^{b_i}}-1= \prod_{j}{g_j(x)}$, where $g_j(x) \in \Q_\ell[x]$ are monic irreducible polynomials. For each such factor $g_j$, we construct an irreducible $\Z_\ell$ representation $\chi_i$ of $\langle\sigma_i\rangle$ as follows. Let $\zeta$ be a root of $g_j$ in an algebraic closure of $\Q_\ell$. Let $\mathcal{O}(\Q_\ell(\zeta))$ denote the ring of integers of $\Q_\ell(\zeta)$, viewed as a $\Z_\ell$ module. Define $\chi_i(\sigma_i)$ to be the $\Z_\ell$ linear automorphism of $\mathcal{O}(\Q_\ell(\zeta))$ that acts on $\mathcal{O}(\Q_\ell(\zeta))$ as left multiplication by $\zeta$. The dimension of the representation equals the degree of the polynomial $g_j$. The fixed space of $\chi_i(\sigma_i)$ is trivial and thus $\chi_i$ is irreducible. Since the factorization $\prod_{j}{g_j(x)}$ is square free and since distinct $g_j$ correspond to distinct representations, counting dimensions reveals that we have constructed all the irreducible representations of $\langle\sigma_i\rangle$.\\ \\
Factor $x^{q_i^{b_i}}-1$ over $\F_\ell$ and lift the factorization to $\Q_\ell[x]$ by Hensel Lifting. For each factor $g_j \in \Q_\ell[x]$, if $\zeta$ is a root of $g_j$ then $\{1,\zeta,\zeta^2,\ldots,\zeta^{deg(g_j)-1}\}$ forms a $\Z_\ell$ basis for $\mathcal{O}(\Q_\ell(\zeta))$. Write down $\chi_i(\sigma_i)$ as the $deg(g_j)$ dimensional square matrix over $\Z_\ell$ that takes the basis $\{1,\zeta,\zeta^2,\ldots,\zeta^{deg(g_j)-1}\}$ to $\{\zeta,\zeta^2,\zeta^3,\ldots,\zeta^{deg(g_j)}=g_j(\zeta)-\zeta^{deg(g_j)}\}$. Thus the irreducible $\Z_\ell$ representation of $G$ and the corresponding idempotents can be determined from the primary decomposition of $G$ as described above and the computation takes time polynomial in $[F:k]$.
\subsection{Computation of the $\ell$-part of the Regulator: Proof of Theorem \ref{regulator_theorem}}\label{regulator}
For every non principal irreducible $\Z_\ell$ representation $\chi$ of $Gal(F/k)$, the regulator part $e(\chi) R_F$ can be computed as follows.  Since $e(\chi)\Z_{\ell} [G]$ is isomorphic to the ring of integers of the unramified abelian extension of $\Q_{\ell}$ of degree $\dim(\chi)$, every simple torsion $e(\chi)\Z_{\ell} [G]$-module is isomorphic to $\Z/\ell^c\Z [G] e(\chi)$ for some $c$. Since $\chi \neq 1$,  the idempotent $e(\chi)$ is orthogonal to the principal idempotent and $e(\chi)  \infty_F$ has degree $0$,  where $\infty_F$ is a place in $F$ above $\infty$.\\ \\
Since $e(\chi) R_F$ is the cyclic $\Z_\ell[G]$ module generated by $e(\chi) \infty_F$, the order of $e(\chi) R_F$ is $\ell^b$ if and only if $b$ is the smallest positive integer for which $\ell^b e(\chi) \infty_F$ is principal. Using an algorithm of Hess \cite{hess}, we can test if  $\ell^b e(\chi)(\infty_F)$ is principal in time polynomial in $[F:k]$, $\log q$ and $\log(\ell^b)$. By finding the smallest $b$ for which it is principal, determine the structure of $e(\chi)R_F$ and theorem \ref{regulator_theorem} follows.\\ \\
Let $b_\chi$ be the smallest positive integer for which $\ell^{b_\chi} e(\chi) \infty_F$ is principal. The cardinality of $e(\chi) R_F$ is $\ell^{b_\chi \dim(\chi)}$ where $\dim(\chi)$ is the dimension of the character $\chi$. From the exact sequence $$ 0 \longrightarrow R_F \longrightarrow  \mathcal{C}l^0_F \longrightarrow \p(\mathcal{O}_F) \longrightarrow 0$$ it follows that $$\left|\p_\ell(\cO_F)\right| = \frac{h(F)_\ell}{\prod_{\chi} \ell^{b_\chi}}$$ where $h(F)_\ell$ is the cardinality of the $\ell$ primary part of $\cl^0_F$ and the product is over all non principal irreducible $\Z_\ell$ representations $\chi$ of $Gal(F/k)$.\\ \\
The divisor class number $h(F)$ can be computed in $\tilde{\mathcal{O}}(p^{12} d^{13} [F:k]^{30})$ time by \cite{lw}[Theorem 37] and theorem ~\ref{ell ideal class number} follows.
\subsection{Computation of the $\ell$-part of the Class Group: Proof of Theorem \ref{class_group_number}}\label{ideal_class_structure}
For this subsection we assume that the field $k$ is the rational function field $\F_q(t)$ and $F=H_\m$ and present the algorithmic details of the iterative procedure outlined in theorem \ref{chi-structure}. At the end of the section we briefly discuss the algorithmic issues involved removing the assumption that $k= \F_q(t)$ and $F=H_\m$.\\ \\
As in the previous section, we fix a non principal irreducible $\Z_\ell$ character $\chi$ of $G$ and set $N= \ell t^2$ where $t$ is the exponent of $(U/E)(\chi)$.\\ \\
To begin the iterative procedure, we need to construct an element in $E$ of $\chi$-index $t$ and an Euler system starting from it.\\ \\
When $k = \F_q(t)$ and $F=H_\m$, $S_F/\mu_F$ is generated by $\{\lambda_m^\sigma | \sigma \in G\}$ and the expression for the Euler system $\Psi_{\f,\g}$ in \S \ref{euler} greatly simplifies. The function $\xi : B_N \longrightarrow k_\infty^\times$ that maps $$\ma \longmapsto \N_{K_{\m\ma}/H_{\ma}}\left(\lambda_\m - \sum_{\pe | \ma}\lambda_\pe \right)  \in H_{\ma}^\times$$ is an Euler system that starts from $\xi(\e) = -\lambda_m^{q-1}$ \cite{fen}[\S 2]. The summation in the above expression is over prime $\pe$ dividing $\ma$.\\ \\
Since we have a finite generating set for $S_F/\mu_F$ and can test for identity in $(E/U^N)(\chi)$, we can compute a basis for $(E/U^N)(\chi)$ and one of the basis elements has to have $\chi$-index $t$. Express this basis element as a product of the form $\frac{\prod_{\sigma \in G_1}\lambda_m^\sigma}{\prod_{\tau \in G_2}\lambda_m^\tau}(e(\chi))$ where $G_1$ and $G_2$ are subsets of $G$ of the same cardinality.\\ \\
Then the function $\Phi$ that maps $\ma \in B_N$ to \begin{equation}\label{euler_system_equation}\frac{\prod_{\sigma \in G_1}\xi(\ma)^\sigma}{\prod_{\tau \in G_2}\xi(\ma)^\tau} \end{equation} is an Euler system starting from an element of $\chi$-index $t$.\\ \\
The iterative algorithm constructs a $\chi$-path starting from $\ka(\e)$,
$$\ka(\e) \stackrel{\chi}{\to} \ka(\pe_1) \stackrel{\chi}{\to} \ka(\pe_1\pe_2) \stackrel{\chi}{\to} \ldots \stackrel{\chi}{\to} \ka(\pe_1 \pe_2 \ldots \pe_n)$$
such that $\chi$-index of $\ka(\pe_1 \pe_2 \ldots \pe_n)$ is $1$.\\ \\
The critical computation at each iteration is to find a $\pe_i$ such that $$\ka(\pe_1 \pe_2 \ldots \pe_{i-1})\stackrel{\chi}{\to} \ka(\pe_1 \pe_2 \ldots \pe_i).$$
The existence of such $\pe_i$ of degree bounded by a polynomial in $N$ and $\log_q([F:k])$ is guaranteed by lemma \ref{density}. Let $m$ be the multiple of $ord_N(q)$ chosen in Lemma \ref{density}. At each step, we randomly generate a prime $\pe_i$ of degree $m$ and check if $\ka(\pe_1 \pe_2 \ldots \pe_n)\stackrel{\chi}{\to} \ka(\pe_1 \pe_2 \ldots \pe_i)$. From proof of lemma \ref{density}, at each step we succeed in finding a $\pe_i$ satisfying  $\ka(\pe_1 \pe_2 \ldots \pe_n)\stackrel{\chi}{\to} \ka(\pe_1 \pe_2 \ldots \pe_{i-1})$ in expected number of trials bounded by a polynomial in $N$.\\ \\
Given a choice of $\pe_i$, we test if $$\ka(\pe_1 \pe_2 \ldots \pe_{i-1})\stackrel{\chi}{\to} \ka(\pe_1 \pe_2 \ldots \pe_i)$$ by first computing $\ka(\pe_1 \pe_2 \ldots \pe_i)$ and then computing its $\chi$ index as described below.\\ \\
We assume that we have constructed $\lambda_{\pe_1},\lambda_{\pe_1},\ldots,\lambda_{\pe_{i-1}}$ in the previous iteration.\\ \\
Compute $\lambda_{\pe_i}$ using lemma \ref{lambda_lemma} and then compute $\Phi(\pe_1\pe_2\ldots\pe_{i})$ using equation \ref{euler_system_equation}.\\ \\
The cyclic extension $F(\pe_1\pe_2\ldots\pe_{i})/F$ can be constructed as the compositum $$F(\pe_1\pe_2\ldots\pe_{i}) = F . H(\pe_1\pe_2\ldots\pe_{i})$$
where $H(\pe_1\pe_2\ldots\pe_{i})$ is the fixed field of $Gal(H_{\pe_1\pe_2\ldots\pe_{i}}/k)^N$.\\ \\
Once $F(\pe_1\pe_2\ldots\pe_{i})$ and $\Phi(\pe_1\pe_2\ldots\pe_{i})$ are constructed, we can compute $\ka(\pe_1\pe_2\ldots\pe_{i})$ using equation \ref{kolyvagin_equation}. The running time for computing $\ka(\pe_1\pe_2\ldots\pe_{i})$ is dominated by the construction of the extension $F(\pe_1\pe_2\ldots\pe_{i})$ and the evaluation of equation \ref{kolyvagin_equation} which take time polynomial in $q^{\deg(\pe_i)}$ and $[F:k]$.\\ \\
All that remains is to compute the $\chi$ index of $\ka(\pe_1\pe_2\ldots\pe_{i})$.\\ \\
To compute the $\chi$-index of an $\alpha \in F^\times/F^{\times N}$, It suffices to be able to decide if $e(\chi)\alpha \in \Z_\ell \otimes_\Z (F^\times/F^{\times N}) $ is an $\ell^{th}$ power. Since $N$ is a power of $\ell$, $\Z_\ell \otimes_\Z (F^\times/F^{\times N}) \cong \Z/N\Z \otimes_\Z (F^\times/F^{\times N})$ and $e(\chi)\alpha$  can be expressed in the form $1 \otimes_\Z f$ and viewed as the function $f$ in $F^\times/F^{\times N}$. Further,  $e(\chi)\alpha$ being an $\ell^{th}$ power in $\Z/N\Z \otimes_\Z (F^\times/F^{\times N})$ is equivalent to $f$ being an $\ell^{th}$ power in $F^\times/F^{\times N}$. Since $|\p_\ell(\cO_F)(\chi)|$ divides $N$, $f$ being an $\ell^{th}$ power in $F^\times/F^{\times N}$ is equivalent to its lift $\hat{f}$ being an $\ell^{th}$ power in $F^{\times}$. The element $\hat{f}$ is an $\ell^{th}$ power in $F^\times$ if and only if the Riemann-Roch space $\mathcal{L}([\hat{f}]/\ell)$ is non empty. We can decide if $\mathcal{L}([\hat{f}]\ell)$ is empty in time polynomial in $[F:k]$ and polylogarithmic in the pole degree of the divisor of $[\hat{f}]/\ell$ \cite{hesrr}.\\ \\
Thus the running time at each iteration of the algorithm is bounded by a polynomial in $q^{ord_N(q)}$ and $[F:k]$.\\ \\
By computing $\p_\ell(\cO_F)(\chi)$ for every $\chi$, we can determine $\p_\ell(\cO_F)$ in time polynomial in $q^{ord_N(q)}$ and $[F:k]$. Since $ord_N(q)$ is at most $N-1$, Theorem \ref{class_group_number} follows.\\ \\
We briefly discuss an issue that arise while attempting to turn Theorem \ref{chi-structure} into an effective algorithm that works not just for $k=\F_q(t)$ and $F=H_\m$ but for every $k,H$ that Theorem \ref{chi-structure} applies to.\\ \\
The generating set $\{\mathcal{N}_{H_\f/H_\f \cap F}(\lambda_\m^{\mathcal{N}(\g)-(\g,K_\m/k)})\}_{\f,\g}$ for the Stark units is not finite since $\f,\g$ range over coprime non zero ideals in $\cO_F$. Hence, it not obvious as to how to find an element in $E$ of $\chi$ index $t$ and construct an Euler system starting from it. In the generating set, the choice of $\f$ can be narrowed to a finite set. It is sufficient to consider $\f$ either dividing $\m$ to account for the ramified part of $F/k$ and $\f$ of bounded degree to account for the unramified part. Since $S_F$ is finitely generated, it should be sufficient to consider $\g$ of bounded degree, but this degree bound needs further investigation.\\ \\
If $k$ is an arbitrary finite geometric extension of $\F_q(t)$ and $F=H_\m$, then an element in $E$ of $\chi$ index $t$ and an Euler system starting from it can be efficiently found using \cite{xu}[Theorem 2.3].
\section{Ackowledgements}
\noindent We would like to thank the two anonymous reviewers for their valuable suggestions.

\end{document}